\documentclass[english,a4paper,10pt,final,usenames, dvipsnames]{amsart}
\usepackage{amsmath,amssymb,amscd,amsxtra,amsgen,enumitem, amsthm,microtype}
\usepackage[all]{xy}
\usepackage[english]{babel}
\usepackage[msc-links,alphabetic,lite,initials]{amsrefs}
\hypersetup{colorlinks=true}
\usepackage{ytableau}

\newcommand{\nc}{\newcommand}
\renewcommand{\frak}{\mathfrak}
\providecommand{\cal}{\mathcal}
\renewcommand{\bold}{\mathbf}

\numberwithin{equation}{section}

\newcommand{\pfname}{Proof.}

                      {$\square$\vskip\medskipamount\par}

\newenvironment{pfof}[1]{\vskip-\lastskip\vskip\medskipamount{\it
    Proof of #1.}}%
                      {$\square$\vskip\medskipamount\par}
\newtheorem{thm}{Theorem}[section]
\newtheorem{theorem}[thm]{Theorem}

\newtheorem{thmnn}{Theorem}

\newtheorem{corollary}[thm]{Corollary}

\newtheorem{prop}[thm]{Proposition}
\newtheorem{proposition}[thm]{Proposition}

\newtheorem{lemma}[thm]{Lemma} 

\theoremstyle{definition}
\newtheorem{definition}[thm]{Definition}

\theoremstyle{definition}

\newtheorem{remark}[thm]{Remark}

\newtheorem{example}[thm]{Example} 

\nc{\Theorem}[1]{Theorem~{#1}}
\nc{\Th}[1]{({\sl Th.}~#1)}
\nc{\Thd}[2]{({\sl Th.}~{#1} {#2})}
\nc{\Theorems}[2]{Theorems~{#1} and ~{#2}}
\nc{\Thms}[2]{({\it Thms. ~{#1} and ~{#2}})}
\nc{\Lemmas}[2]{Lemma~{#1} and ~{#2}}

\nc{\manga}[6]{({\it Thms. ~ #1, ~ #2, ~ #3,\\ ~ #4, ~ #5, ~ #6})}

\nc{\Prop}[1]{({\sl Prop.}~{#1})}
\nc{\Proposition}[1]{Proposition~{#1}}
\nc{\Propositions}[2]{Propositions~{#1} and ~{#2}}
\nc{\Props}[2]{({\sl Props.}~{#1} and ~{#2})}
\nc{\Cor}[1]{({\sl Cor.}~{#1})}
\nc{\Corollary}[1]{Corollary~{#1}}
\nc{\Corollaries}[2]{Corollaries~{#1} and ~{#2}}
\nc{\Definition}[1]{Definition~{#1}}
\nc{\Defn}[1]{({\sl Def.}~{#1})}
\nc{\Lemma}[1]{Lemma~{#1}} 
\nc{\Lem}[1]{({\sl Lem.} ~{#1})} 
\nc{\Eq}[1]{equation~({#1})}
\nc{\Equation}[1]{Equation~({#1})}
\nc{\Section}[1]{Section~{#1}}
\nc{\Sections}[1]{Sections~{#1}}
\nc{\Sec}[1]{({\sl Sec.} ~{#1})} 
\nc{\Chapter}[1]{Chapter~{#1}}
\nc{\Chapt}[1]{({\sl Ch.}~{#1})}

\nc{\Ex}[1]{{\sl Ex.}~{#1}}
\nc{\Exa}[1]{{\sl Example}~{#1}}
\nc{\Example}[1]{{\sl Example}~{#1}}
\nc{\Examples}[1]{{\sl Examples}~{#1}}
\nc{\Exercise}[1]{{\sl Exercise}~{#1}}

\nc{\Rem}[1]{({\sl Rem.}~{#1})}
\nc{\Remark}[1]{{\sl Remark}~{#1}}
\nc{\Remarks}[1]{{\sl Remarks}~{#1}}
\nc{\Note}[1]{{\sl Note}~{#1}}

\nc{\Conjecture}[1]{Conjecture~{#1}}
\nc{\Claim}[1]{Claim~{#1}}

\nc \Proof{{  \it Proof. }}

\nc \Dbb{\mathbb D}
\nc{\xmu}{\mu}
\nc{\w}{\omega}
\nc{\xv}{\mbox{\boldmath$x$}}
\nc{\uv}{\mbox{\boldmath$u$}}
\nc{\xiv}{\mbox{\boldmath$\xi$}}
\nc{\bbeta}{\mbox{\boldmath$\beta$}}
\nc{\balpha}{\mbox{\boldmath$\alpha$}}
\nc{\bgamma}{\mbox{\boldmath$\gamma$}}
\nc{\bdelta}{\mbox{\boldmath$\delta$}}
\nc{\bepsilon}{\mbox{\boldmath$\epsilon$}}

\nc \Ab{{\ensuremath{\bold A}}}
\nc \ab{{\ensuremath{\bold a}}}
\nc \bb{{\ensuremath{\bold b}}}
\nc \cb{{\ensuremath{\bold c}}}
\nc \db{{\ensuremath{\bold d}}}
\nc \Bb{{\ensuremath{\bold B}}}
\nc \Gb{{\ensuremath{\bold G}}}
\nc \Qb{{\ensuremath{\bold Q}}}
\nc \Rb{{\ensuremath{\bold R}}} \nc \Cb{{\ensuremath{\bold C}}} 
\nc \Eb{{\ensuremath{\bold E}}}
\nc \eb{{\ensuremath{\bold e}}}
\nc \Db{{\ensuremath{\bold D}}}
\nc \Fb{{\ensuremath{\bold F}}}
\nc \ib{{\ensuremath{\bold i}}}
\nc \jb{{\ensuremath{\bold j}}}
\nc \kb{{\ensuremath{\bold k}}}
\nc \Kb{{\ensuremath{\bold K}}}
\nc \nb{{\ensuremath{\bold n}}}
\nc \rb{{\ensuremath{\bold r}}}
\nc \Ob{{\ensuremath{\bold O}}}
\nc \Pb{{\ensuremath{\bold P}}}
\nc \pb{{\ensuremath{\bold p}}}
\nc \qb{{\ensuremath{\bold q}}}
\nc \SPb{{\ensuremath{\bold {SP}}}}
\nc \Zb{{\ensuremath{\bold Z}}} 
\nc \zb{{\ensuremath{\bold z}}} 
\nc \gb{{\ensuremath{\bold g}}} 
\nc \fb{{\ensuremath{\bold f}}} 
\nc \ub{{\ensuremath{\bold u}}} 
\nc \vb{{\ensuremath{\bold v}}} 
\nc \yb{{\ensuremath{\bold y}}} 
\nc \xb{{\ensuremath{\bold x}}} 
\nc \tb{{\ensuremath{\bold t}}} 
\nc \Xb{{\ensuremath{\bold X}}} 
\nc \Yb{{\ensuremath{\bold Y}}} 
\nc \xib{{\ensuremath{\bold \xi}}} 
\nc \Nb{{\ensuremath{\bold N}}} 
\nc \Hb{{\ensuremath{\bold H}}} 
\nc \wb{{\ensuremath{\bold w}}} 
\nc \Wb{{\ensuremath{\bold W}}} 
\nc \syz{{\mathbf {syz}}}
\nc \bnoll{{\ensuremath{\bold 0}}} 

\nc \mf{\frak m} 
\nc \mh{\hat{\mf}} 
\nc \nf{\frak n}
\nc \Of{\frak O}
\nc \of{\frak o}
\nc \rf{\frak r}
\nc \tf{\frak t}
\nc \mufr{{\mathbf \mu}}
\nc \hf{\frak h} 
\nc \qf{\frak q} 
\nc \bfr{\frak b} 
\nc \kfr{\frak k} 
\nc \pfr{\frak p} 
\nc \af{\frak a }
\nc \cf{\frak c }
\nc \sfr{\frak s} 
\nc \ufr{\frak u} 
\nc \g{\frak g} 
\nc \gA{\g_{\Ao}} 
\nc \lfr{\frak l}
\nc \afr{\frak a}
\nc \gfh{\hat {\frak g}}
\nc \gl{\frak { gl }}
\nc \Sl{\frak {sl}}
\nc \SU{\frak {SU}}
\nc{\Homf}{\frak{Hom}}

\newcommand{\on}{\operatorname}

\nc\hankel{\on {Hankel}}
\nc\row{\on {row\ }}
\nc\nullity{\on {nullity }}
\nc\col{\on {col\ }}
\nc\rowm{\on {Row \ }}
\nc\loc{\on {lc \ }}
\nc\nullo{\on {null\ }}
\nc\Nul{\on {Nul\ }}
\nc \Ann {\on {Ann }}
\nc \Ass {\on {Ass}}
\nc \Coker {\on {Coker}}
\nc \Co{\on C}
\nc \Homo{\on {Hom}}
\nc \Ker {\on {Ker}}
\nc \omod{\on {mod}}
\nc \No {\on N}
\nc \NN {\on {NN}}
\nc \NGo {\on {NG}}
\nc \Oo {\on O}
\nc \ch {\on {ch}}
\nc \rko {\on {rk}}
\nc \Sing {\on {Sing\ }}
\nc \Reg {\on {Reg}}
\nc \CoI {\on {CI}}
\nc \CoM {\on {CM}}
\nc \Gor {\on {Gor}}
\nc \Type {\on {Type}}
\nc \can {\on {can}}
\nc \Top {\on {T}}
\nc \Tr {\on {Tr}}
\nc \rel {\on {rel}}
\nc \tr {\on {tr}}
\nc \sgn {\on {sgn }}
\nc \trdeg {\on {tr.deg}}
\nc \codim {\on {codim }}
\nc \coht {\on {coht}}
\nc \divo {\on {div \ }}
\nc \rot {\on {rot }}
\nc \coh {\on {coh}}
\nc \Clo {\on {Cl}}
\nc \embdim{\on {embdim}}
\nc \ord{\on {ord}}
\nc \ed{\on {ed}}
\nc \embcodim{\on {embcodim  }}
\nc \qcoh {\on {qcoh}}
\nc \grad {\on {grad}}
\nc \grade {\on {grade}}
\nc \hto {\on {ht}}
\nc \depth {\on {depth}}
\nc \prof {\on {prof}}
\nc \reso{\on {res}}
\nc \ind{\on {ind}}
\nc \prodo{\on {prod}}
\nc \coind{\on {coind}}
\nc \Con{\on {Con}}
\nc \Crit{\on {Crit}}
\nc \Der{\on {Der}}
\nc \Des{\on {Des}}
\nc \Char{\on {Char}}
\nc \Ch{\on {Ch}}

\nc \Ext{\on {Ext}}
\nc \Eo{\on {E}}
\nc \End{\on {End}}
\nc \ad{\on {ad}}
\nc \Ad{\on {Ad}}
\nc \gr{\on {gr}}
\nc \Fo{\on {F}}
\nc \Gr{\on {Gr}}
\nc \Go{\on {G}}
\nc \GFo{\on {GF}}
\nc \Glo{\on {Gl}}
\nc \PGlo{\on {PGl}}
\nc \Ho{\on {H}}
\nc \CMo{\on {\CM}}
\nc \SCM{\on {SCM}}
\nc \rig{\on {right}}
\nc \lef{\on {left}}
\nc \hol{\on {hol}}
\nc{\sgd}{\on{sgd}}
\nc \supp{\on {supp}}
\nc \ssupp{\on {s-supp}}
\nc \singsupp{\on {singsupp}}
\nc \msupp{\on {msupp}}
\nc \spec{\on {spec}}
\nc \spano{\on {span }}
\nc \Span{\on {Span }}
\nc \Max{\on {Max}}
\nc \Mat{\on {Mat}}
\nc \Min{\on {Min}}
\nc \nil{\on {nil}}
\nc \Mod{\on {Mod}}
\nc \Rad {\on {Rad}}
\nc \rad {\on {rad}}
\nc \rank {\on {rank}}
\nc \range {\on {range}}
\nc \Slo{\on {SL}}
\nc \soc {\on {soc}}
\nc \dt {\on {dt}_Z}
\nc \Irr {\on {Irr}}
\nc \Reo {\on {Re}}
\nc \Imo {\on {Im}}
\nc \SSo{\on {SS}}
\nc \lub{\on {lub}}
\nc \gldim{\on {gl.d.}}
\nc \length{\on {length}}
\nc \pdo{\on {p.d.}} 
\nc \fdo{\on {f.d.}} 
\nc \ido{\on {i.d.}} 
\nc \dSSo{\dot {\SSo}}
\nc \So{\on S}
\nc \SOo{\on{ SO}}
\nc \Io{\on I}
\nc \Jo{\on J}
\nc \jo{\on j}
\nc \Ko{\on K}
\nc \PBW{\Ac_{PBW}}
\nc \Ro{\on R}
\nc \To{\on T}
\nc \Ao{\on A}

\nc \Do{{\on D}}
\nc \Bo{\on B}
\nc \Po{\on P}
\nc \Qo{\on Q}
\nc \Zo{\on Z}
\nc \Uo{\on U}
\nc \wt{\on {wt}}
\nc \Uoh{\hat {\Uo}}
\nc \Lo{\on L}
\nc \Loc{\on {Loc}}
\nc{\dop}{\on d}
\nc{\eo}{\on e}
\nc{\ado}{\on{ad}}
\nc{\Tot}{\on{Tot}}
\nc{\Aut}{\on{Aut}}
\nc{\sinc}{\on {sinc}}

\nc{\overrightleftarrows}[2]{\overset{#1}{\underset{#2}{\rightleftarrows}}}

\nc{\CCF}{\cal{CF}}
\nc{\CDF}{\cal{DF}}
\nc{\CHC}{\check{\cal C}}

\nc{\Cone}{\on{Cone}}
\nc{\dec}{\on{dec}}
\nc{\Diff}{\on{Diff}}
\nc{\dirlim}{\underset{\to}{\on{lim}}}
\nc{\dpar}{\partial}
\nc{\GL}{\on{GL}}
\nc{\glo}{\on{gl}}
\nc{\CGr}{\cal{G}r}
\nc{\pr}{\on{pr}}
\nc{\semid}{|\!\!\!\times}
\nc{\Hom}{\on{Hom}}
\nc \RHom{\on {RHom}}

\nc \Proj{\mathrm {Proj\ }}
\nc \proj{\mathrm {proj}}
\nc{\Id}{\on{Id}}
\nc{\id}{\on{id}}
\nc{\Ima}{\on{Im}}
\nc{\invtimes}{\underset{\gets}{\otimes}}
\nc{\invlim}{\underset{\gets}{\on{lim}}}
\nc{\Lie}{\on{Lie}}
\nc{\re}{\on{Re }}
\nc{\Pic}{\on{Pic }}
\nc{\LPic}{\on{LPic }}
\nc{\Sch}{\on{Sch}}
\nc{\Sh}{\on{Sh}}
\nc{\Set}{\on{Set}}
\nc{\spo}{\on{sp\  }}
\nc{\Spec}{\on{Spec}}
\nc{\mSpec}{\on{mSpec}}
\nc{\Specb}{\bold {Spec}\ }
\nc{\Projb}{\bold {Proj}}
\nc{\Specan}{\on{Specan}}
\nc{\Spo}{\on{Sp}}
\nc{\Mpo}{\on{Mp}}
\nc{\Spf}{\on{Spf}}
\nc{\sym}{\on{sym}}
\nc{\symm}{\on{symm}}
\nc{\rop}{\on{r}}
\nc{\Td}{\on{Td}}
\nc{\Tor}{\on{Tor}}

\nc{\Alg}{\on {Alg}}
\nc{\Artin}{\cal{A}rtin}
\nc{\Dgcoalg}{\cal{D}gcoalg} \nc{\Dglie}{\cal{D}glie}
\nc{\Ens}{\cal{E}ns} \nc{\Fsch}{\cal{F}sch}
\nc{\Groupoids}{\cal{G}roupoids}
\nc{\Holie}{\cal{H}olie}
\nc{\Mor}{\cal{M}or}

\nc \Dd{\mathbb D}
\nc{\CF}{\ensuremath{\cal{F}}}
\nc \Kc{{\ensuremath{\cal K}}}
\nc{\Kzind}[4]{{\ensuremath{{\mathcal K^{#4e} (#3)}_{#1 \rightarrow
        #2}}}}
\nc{\Kz}[3]{{\ensuremath{{\mathcal K^\bullet (#3)}_{#1 \rightarrow
        #2}}}}
\nc{\Kzd}[2]{{\ensuremath{{\mathcal K}^\bullet_{#1 \rightarrow #2}}}}
\nc \Lc{{\ensuremath{\cal L}}}
\nc \lcc{{\mathcal l}} 
\nc \CC{{\ensuremath{\cal C}}} 
\nc \Cc{{\ensuremath {\cal C}}}
\nc \Pc{{\ensuremath{\cal P}}}
\nc{\Dl}[2]{{\ensuremath{{\mathcal D}_{#1 \leftarrow #2}}}}
\nc{\Dr}[2]{{\ensuremath{{\mathcal D}_{#1 \rightarrow #2}}}}
\nc \Dc{\ensuremath{\mathcal D}}
\nc \Ac{{\ensuremath{\cal A}}} 
\nc \Bc{{\ensuremath{\cal B}}}
\nc \Ec{{\ensuremath{\cal E}}}
\nc \Fc{{\ensuremath{\cal F}}}
\nc \Mcc{{\ensuremath{\cal M}}} 
\nc \hM{\hat{\Mcc}} 
\nc \bM{\bar {\Mcc}} 
\nc\hbM{\hat{\bar \Mcc}}  
\nc \Nc{{\ensuremath{\cal N}}}
\nc \Hc{{\ensuremath{\cal H}}} 
\nc \Ic{{\ensuremath{\cal I}}} 
\nc \Jc{{\ensuremath{\cal J}}} 
\nc \Oc{\ensuremath{{\cal O}}}
\nc \Och{\hat{\cal O}} 
\nc \Sc{{\ensuremath{{\cal S}}}}
\nc \Tc{\ensuremath{{\cal T}}} 
\nc \Vc{{\ensuremath{{\cal V}}}} 
\nc{\CA}{{\ensuremath{{\cal A}}}}
\nc{\CB}{{\ensuremath{{\cal B}}}}
\nc{\Fcc}{\ensuremath{{\cal F}}}
\nc{\Gc}{{\ensuremath{{\cal G}}}}
\nc{\CH}{\ensuremath{\mathcal H}}
\nc{\CI}{{\ensuremath{{\cal I}}}}
\nc{\CM}{{\ensuremath{{\cal M}}}}
\nc{\CN}{{\ensuremath{{\cal N}}}}
\nc{\CO}{{\ensuremath{{\cal O}}}}
\nc{\Rc}{{\ensuremath{{\cal R}}}}
\nc{\CT}{{\ensuremath{\mathcal T}}}
\nc{\CU}{\ensuremath{{\cal U}}}
\nc{\Uc}{\ensuremath{{\cal U}}}
\nc{\Yc}{\ensuremath{{\cal Y}}}
\nc{\CV}{\ensuremath{{\cal V}}}
\nc{\CZ}{\ensuremath{{\cal Z}}}
\nc{\Homc}{\ensuremath{{\cal {Hom}}}}

\nc{\fa}{\frak{a}}
\nc{\fA}{\frak{A}}
\nc{\fg}{\frak{g}}
\nc{\fh}{\frak{h}}
\nc{\fI}{\frak{I}}
\nc{\fK}{\frak{K}}
\nc{\fm}{\frak{m}}
\nc{\fP}{\frak{P}}
\nc{\fS}{\frak{S}}
\nc{\ft}{\frak{t}}
\nc{\fX}{\frak{X}}
\nc{\fY}{\frak{Y}}

\nc{\bF}{\bar{F}}
\nc{\bCP}{\bar{\cal{P}}}
\nc{\bmbox}{\mbox{\bf{m}}}
\nc{\bT}{\mbox{\bf{T}}}
\nc{\hB}{\hat{B}}
\nc{\hC}{\hat{C}}
\nc{\hP}{\hat{P}}
\nc{\htest}{\hat P}

\nc{\nen}{\newenvironment}
\nc{\ol}{\overline}
\nc{\unl}{\underline}
\nc{\ra}{\to}
\nc{\lla}{\longleftarrow}
\nc{\lra}{\longrightarrow}
\nc{\Lra}{\Longrightarrow}
\nc{\Lla}{\Longleftarrow}
\nc{\Llra}{\Longleftrightarrow}
\nc{\hra}{\hookrightarrow}
\nc{\iso}{\overset{\sim}{\lra}}

\nc{\dsize}{\displaystyle}
\nc{\sst}{\scriptstyle}
\nc{\tsize}{\textstyle}
\nen{exa}[1]{\label{#1}{\bf Example.\ } }{}

\nen{rem}[1]{\label{#1}{\em Remark.\ } }{}
\nen{exer}[1]{\label{#1}{\em Exercise.\ } }{}

\theoremstyle{definition}

\theoremstyle{remark}

\nc{\Sats}[1]{Sats~\ref{#1}}
\nc{\Sa}[1]{({\sl Sa.}~\ref{#1})}
\nc{\Kor}[1]{({\sl Kor.}~\ref{#1})}
\nc{\Korollarium}[1]{Korollarium~\ref{#1}}
\nc{\Exe}[1]{{\sl Exempel}~\ref{#1}}
\nc{\Anm}[1]{{\sl Anmärkning}~\ref{#1}}
\nc{\Not}[1]{{\sl Not}~\ref{#1}}
\begin{document}

\title[Decomposition over invariant differential operators ]{Decomposition of modules over invariant differential operators}
\author{Rikard B{\o}gvad}
\address{Department of Mathematics, Stockholm University}
\email{rikard@math.su.se}
\author{Rolf K{\"a}llstr{\"om}}
\address{Department of Mathematics, University of G{\"a}vle}
\email{rkm@hig.se}
\date \today \maketitle
\tableofcontents
\section{Introduction} 
Given a finite subgroup $G\subset \Glo(V)$ of the linear group of a finite-dimensional complex
vector space $V$, it is a well-studied problem to describe the structure of the symmetric algebra
$B=\So(V)$ as a representation of $G$, and also as a module over the ring of invariant differential
operators $\Dc= \Dc_B^G \subset \Dc_B$ in the ring of differential operators on $B$, where we
mention in passing that $\Dc$ is also the set of liftable differentiable operators with respect to
the map $B^G \to B$ (see \cite{knop:gradcofinite}). In fact, the two perspectives are known to be
equivalent; for a precise statement, see \Proposition{\ref{thm:montgomery}}. The ring $\Dc$ inherits
the natural grading of $B$, and we let $\Dc^0\subset \Dc$ and $\Dc^-\subset \Dc$ be the invariant
differential operators of degree $0$ and strictly negative degree, respectively. Our first and main
result is that there is for all such finite groups a ``lowest weight'' description of the category
of $\Dc$-submodules of $B$, where the ring $\Rc=\Dc^0/(\Dc^0\cap (\Dc \Dc^-))$ plays the role of
``Cartan algebra''. \newtheorem{thmlabel}{\bf Theorem} \renewcommand{\thethmlabel}{\ref{main}}
\begin{thmlabel}\label{maintheorem}   The functor
$$
N\mapsto N^{ann}=\Ann_{\Dc^-}(N) = \{n \in N \ \vert \ \Dc^- n=0\}$$
is an equivalence between the category of $\Dc$-submodules of $B$ and the category of
$\Rc$-submodules of $B^{ann}$.
\end{thmlabel}
The $\Rc$-module $B^{ann}$ is also a subrepresentation of the space of
$G$-harmonic polynomials, and is therefore finite-dimensional. As an
example of what the theorem contains, we mention that for $G$ a
generalized symmetric group, $\Rc$ is a quotient of the commutative
algebra $R_n= \Cb [x_1\partial_1, \ldots , x_n\partial_n]^{S_n}$. As a
rather immediate consequence, isomorphism classes of simple
$\Rc$-modules are 1-dimensional and classified by (ordered)
partitions, hence by \Theorem{\ref{maintheorem}} the same
classification applies to the simple $\Dc$-submodules of $B$ as well
as to all representations of $G$. This is of course well-known but
here it is a consequence of the explicit structure of $R_n$.

The $G$-representation $B^{ann}$, which contains a copy of all irreducibles
\Prop{\ref{thm:montgomery}}, has been studied under the name of the polynomial model
\citelist{\cite{aguado-araujo:symmetric} \cite{gelfand.garge}}, in particular when $G$ is a complex
reflection group, with the aim to determine when each irreducible occurs with multiplicity 1; one
then says that $B^{ann}$ is a Gelfand model.  The $\Rc$-structure, however, seems not to have been
exploited, in spite of the fact that the above theorem has the following nice immediate consequence,
just using the fact that simple modules over commutative $\Cb$-algebras are one-dimensional:
\renewcommand{\thethmlabel}{\ref{thm:rcomm}}
\begin{thmlabel}\label{gelfandcor}  
If $\Rc$ is commutative then $B^{ann}$ is a Gelfand model for $G$.
\end{thmlabel}

As already mentioned $\Rc$ is commutative for $G$ a generalized symmetric group $G(d,1,n)$ (which
includes all Weyl groups of type $A$ and $B$), as well as when $G$ is a dihedral group. Hence we
have in particular a short and conceptual proof that $B^{ann}$ is a Gelfand model for these groups,
a result due to \cite{generalized} when $G=G(d,1,n)$. Several authors have attempted to construct Gelfand
models for $G(de,e,n)$ with $e>1$, and it might be hoped that a study of $\Rc$ in this case would be
similarly helpful.

One way of computing $\Dc^-$ and $\Rc$ is by utilizing the strong result by Levasseur and Stafford
\cite{levasseur-stafford:invariantdiff} that $\Dc$  is generated as an algebra by its two commutative
subrings $B^G= \So(V)^G$ and $\So (V^*)^G$, where $V^*=\oplus_{i=1}^n\Cb \partial_i$ is the vector
space of constant derivations and $V=\oplus_{i=1}^n\Cb x_i$.  Ring generators $f_i(x)$ of
$\So(V)^G $ and $f_i(\partial)\in \So (V^*)^G $ thus give generators of $\Dc$, but they also
generate a Lie subalgebra $\af\subset \Dc$, for which the PBW-theorem then is available.  In the
case of the generalized symmetric group a good choice is to let $f_i$ be power sums, which gives a
basis of $\af$ by elements of the form $\sum_{i=1}^n x_i^k \partial_i^l$, which we call power
differential operators; these operators turn out to be amenable to effective computation. For our
calculations with the dihedral group we use a different and more straightforward technique to get
$\Dc^-$ and $\Rc$.

There is another context in which $B^{ann}$ occurs, though only
implicitly, and without using differential operators, namely that of
Macdonald-Spaltenstein-Lusztig induction of representations relative
to an inclusion of finite groups $H\subset G$. In fact, this induction
is best understood as an operation on $\Dc$-modules, described in
\Theorem{\ref{prop:MLS}} (which relies on
\Theorem{\ref{maintheorem}}), instead of $G$-modules; for the relation
with the usual definition for groups, see
\Proposition{\ref{prop:MLSrelation}}. In our differential algebra
context, MLS-restriction (instead of induction) will be
\begin{displaymath}
 \Jo_H^G: \Mod_{\Dc_1}(B)\to \Mod_{\Dc_2}(B),\quad N \mapsto \Jo_H^G (N)=\Dc_2\cdot \Ann_{\Dc_1^-}(N), 
\end{displaymath}
where $\Dc_2=\Dc_B^G\subset \Dc_1=\Dc_B^H$. We exemplify the use of
$\Jo_H^G$ by constructing the simple components of the $\Dc_2$-module
$B$ when $G$ is the generalized symmetric group.

Another application of \Theorem{\ref{maintheorem}} is to get an
abstract branching rule (multiplicity 1) \Th{\ref{branch-diff}},
exemplified with generalized symmetric groups, and a rather detailed
decomposition of restricted modules for the symmetric group
\Th{\ref{branch2}}, providing a new proof of the classical branching
rule using lowest weight arguments.

Our last application is to a new construction of Young bases for
representations of the symmetric group, showing the close relation
between the Jucys-Murphy elements $L_i= \sum_{j=1}^{i-1} (j \ i)$ (a
sum of transpositions) in the group algebra of $S_n$ and the nabla
operators $\nabla_i=x_i\partial_i$. Put $\Dc_n= \Dc_B^{S_n}$ and
$B_i^{ann}= \Ann_{\Dc_i^-}(B)$, $1\leq i \leq n$. Given a basis
$\{v_j^{(i-1)}\}$ of $B^{ann}_{i-1}$, by the branch rule the
$R_i$-module $\Ann_{\Dc_i^-}(\Dc_{i-1}v_j^{(i-1)})$ is multiplicity
free, so a decomposition into simples gives a basis. Iterating this
procedure one gets a {\it canonical basis} $\{v_T\}_{T\in \Sc}$ of
$B^{ann}$, indexed by the set of paths in a branching graph, which in
turn are encoded by standard Young tableux. Interestingly enough, it
turns out that the canonical basis is the same as the Young basis
\Th{\ref{can-base}}.
The weights of the commutative algebra generated by the $L_i$ that is
used in \cite{okounkov-vershik} here has a natural and more immediate
analogue in the multidegree of nabla operators. One can conclude from
\Theorem{\ref{branch2}} and \Theorem{\ref{can-base}} that it is
possible to build up the representation theory of the symmetric group
from the action of nabla operators in the ring of differential
operators.

  
In the final section we study the dihedral group $D_{2e}$ of order $2e$ acting on
$\Cb^2$. Noteworthy is the fact that $\Rc$ for its cyclic subgroup $C_{e}$ is non-commutative
(though still simple to describe), that moreover in this case the lowest weight space $N^{ann}$ of a
certain simple module $N$ may have dimension strictly larger than 1, and that for this module
MLS-restriction $\Jo^G_H$ does not preserve simplicity, where $H=C_{e}$ and $G= D_{2e}$.

We conclude by the remark that though most of our examples are taken
from reflection groups, they serve primarily as examples of the use of
the setup. This setup, however, is valid quite generally, and we
suspect it is worthwhile, e.g., to compute $\Rc$ for other groups.



\section{Preliminaries}
\subsection{Notation}
\label{section:notation}
We will throughout the paper assume that we have a finite subgroup $G$ of the general linear group
$\Glo(V)$ of a complex finite dimensional vector space $V$, inducing a graded action on the graded
polynomial algebra $B=\So(V)$ (with $V$ in degree $1$).  The algebra of differential operators on
$B$ is denoted by $\Dc_B$ (and sometimes $\Dc(V)$), which is the Weyl algebra in $n=\dim_{\Cb} V$
variables.  The canonical map $V\otimes_\Cb V^*\to \Cb$ can be extended to an isomorphism of left
$\So(V)$-module and right $\So(V^*)$-module (not as rings)
\begin{displaymath}\tag{*}
\So(V)\otimes_{\Cb}\So(V^*) \to \Dc_B, \quad p\otimes q \mapsto (b \mapsto   p(x_1, \dots , x_n) q(\partial_1, \dots , \partial_n)(b)),
\end{displaymath}
where $x_1,\ldots, x_n$ is a basis of $V$, $\partial_1,\ldots , \partial_n$ its dual basis of $V^*$, and
$q(\partial_1, \dots , \partial_n)(b)$  the usual action of a constant coefficient differential
operator an a polynomial $b$.  Note that as a Lie sub algebra of $\Dc_B$ (with the commutator as
bracket) the homogeneous derivations can be identified with the general Lie algebra
$V\otimes_{\Cb} V^*=\glo(V)$, and that this Lie subalgebra contains the canonical element
$\nabla=x_1\partial_1+\cdots +x_n\partial_n$. The adjoint action of $\nabla$ on $\Dc_B$ gives a
decomposition $\Dc_B= \oplus \Dc_B(n)$, where
$\Dc_B(n) = \{P \in \Dc_B \ \vert \ [\nabla, P]= nP\}$; it gives $\Dc_B$ the same grading as the
natural one that is induced by the identification \thetag{*}, placing $V^*$ in degree $-1$
and $V$ in degree $1$.

There is an induced action of $G$ on $\Dc_B$ that can be described
using \thetag{*}, as coming from the canonical left action on
$V^*$ and $V$.  The algebra of invariant differential operators
$\Dc = \Dc_B^G$ naturally acts on the invariant ring $A=B^G$, so there
is a homomorphism $\Dc\to \Dc_A$. 

\subsection{Nabla operators} When $\dim V=1$ the above construction
gives us the Weyl algebra $\Dc(\Cb) =\Cb[x,\partial]$ in 1
variable. Its subspace $\Dc(\Cb)^0= \Dc(\Cb)(0)$ of degree $0$ has the
basis $\{x^i\partial^i\}_{i\geq 0}$, where we in particular have the
canonical element $\nabla= x\partial $.  The following easy result
concerned with $\nabla$ will prove useful.
\begin{lemma}
\label{lemma:nabla}
\begin{enumerate}
\item $\Dc(\Cb)^0=\Cb[\nabla]$. In particular there are polynomials
  $p_k\in \Cb[t], \ k=0,1,2,\ldots$, such that $x^k\partial^k=p_k(\nabla)$.
\item $[\nabla,x^i\partial^j]=(i-j)x^i\partial$.
\item Assume that $[\nabla,v]=av$, where $v\in V$. Then, for any polynomial $p(t)\in \Cb [t]$,
  $ p(\nabla)v=vp(\nabla+a)$ and consequently also $p(\nabla-a)v=vp(\nabla)$.
\end{enumerate}
\end{lemma}

\subsection{Representations of groups and $\Dc$-modules}
The group algebra $\Dc [G]$ of $G$ with coefficients in $\Dc$ consists
of functions $\sum_{g\in G} P_g g : G \to \Dc$, $g\mapsto P_g$ , where
the product is
\begin{displaymath}
  \sum_{g_1\in G} P_{g_1} g_1 \cdot \sum_{g_2\in G} Q_{g_2} g_2 =
  \sum_{g\in G}\sum_{g_1, g_2= g} (P_{g_1} Q_{g_2}) g.
\end{displaymath}
Then $B$ is a $\Dc [G]$-module.  Recall also that if $M$ is a
semi-simple module over a ring $R$, and $N$ is an simple $R$-module,
then the isotypic component $M_N$ of $M$ associated to $N$ is the sum
$\sum N'\subset M$ of all $N'\subset M$ such that $ N'\cong N$. Let
$\hat G$ denote the set of isomorphism classes of irreducible complex
$G$-representations.

\begin{proposition}
\label{thm:montgomery}
As a $\Dc[G]$-module, 
we have a decomposition into simple submodules 
\begin{displaymath}
  B= \bigoplus_{\chi\in \hat G}B_\chi,
  \end{displaymath}
 where each simple $B_\chi$ occurs with multiplicity one.  
   \begin{enumerate}
  \item This decomposition coincides with the decomposition of $B$ into isotypic components either
    as a representation of  $G$ or as a $\Dc$-module.
  \item If $B_\chi$ is the isotypic component of the irreducible $G$-representation $V_\chi$ and of the
    simple $\Dc$-module $N_\chi$, respectively, then, as a $\Dc[G]$-module,
 \begin{displaymath}
B_{\chi}\cong V_\chi\otimes_{\Cb} N_\chi.
\end{displaymath}
 (Here the action on the right is given by $(gP)(v\otimes n)=gv\otimes Pn,\ g\in G,\ P\in \Dc $). 
\item In the situation in (2), $$N_\chi \cong Hom_G(V_\chi, B),$$ as a
  $\Dc$-module, and $$V_\chi\cong Hom_{\Dc}(N_\chi,B),$$as a representation of
  $G$.
\end{enumerate}
\end{proposition}
\begin{proof} These results, though parts occur in
  \cite{montgomery:fixed}, may be found in \cite[Lemma 3.3 and Thm.
  3.4]{levasseur-stafford:invariantdiff} and \cite[Prop.1.5 and Thm.
  1.6]{wallach:invariantdiff}.
\end{proof}

If the isotypic component of an irreducible $G$-representation $V$ in $B$ coincides with the
isotypic component of the $\Dc$-module $N$, as in (2) above, we will write $N\sim_G V$.  Note that
as a direct corollary of (2), the isotypic component corresponding to a linear
character $\phi: G\to\Cb^*$ is in itself a simple $\Dc$-module. In this case the isotypic component
is called the {\it module of semi-invariants} \/ associated to $\phi$.  The above results may also be
viewed as consequences of the decomposition theorem of direct images in D-module theory,
\cite{kallstrom-directimage}.

Let $H$ be a subgroup of $G$, so that $\Dc_B^G \subset \Dc^H_B$.  For a $\Dc^{H}$-submodule $N$ of
 $B$, we let $\reso^{\Dc^G_B}_{\Dc_B^H} (N) = N$, where $N$ is considered as $\Dc^G_B$-module by
 restriction to the subring. For an $H$-representation $V$ we let
 $\ind^G_H V = \Cb[G]\otimes_{\Cb[H]} V$ be the induced representation of $G$.
 \begin{proposition}\label{ind-res} Assume that $V\sim_ H N $ in the correspondence
   \Proposition{\ref{thm:montgomery}}, (3) (with $G=H$). Then
   \begin{displaymath}
     \reso^{\Dc^G_B}_{\Dc_B^H} (N) \sim_G \ind^G_H V.
   \end{displaymath}
 \end{proposition}
 \begin{proof}
   Put $\Loc_G (V) = Hom_G(V,B)$ and $\Delta_H (N) = Hom_{\Dc_B^H}(N,B)$, so that $\Loc_G(V)$ is a
   $\Dc_B^G$-module and $\Delta_H(N)$ is an $H$-representation. Then
   \begin{align*}
     &\Loc_G \circ \ind^G_H \circ \Delta_H(N) = Hom_G(\Cb[G]\otimes_{\Cb[H]} Hom_{\Dc_B^H}(N,B),B)\\ & =
                                                                                                       Hom_H(Hom_{\Dc_B^H}(N,B), Hom_G(\Cb[G], B)) =
                                                                                                       Hom_H(Hom_{\Dc_B^H}(N,B),B)
                                                                                                       = N,
   \end{align*}
where $N$ is  only regarded as a $\Dc_B^G$-module. 
 \end{proof}




\section{Equivalence between $\Dc$-modules and $\Dc^0$-modules}
In (\ref{abstract-eq}) and (\ref{inv-rings}) we present our main result, which is about studying
$\Dc$-modules $M$ by its lowest weight space $\Ann_{\Dc^-}(M)$, where the latter is a module over
$\Rc= \Dc^0/(\Dc^0\cap (\Dc \Dc^-))$. In (\ref{section:d0}) we work out methods to compute $\Rc$ and
$\Dc^-$, which are also exemplified. Gelfand models are discussed in (\ref{gelfand-section}).
\subsection{Abstract equivalence}\label{abstract-eq}
We describe the equivalence first in a more general setting than we need, to facilitate the proof
and to give a model that perhaps can be used elsewhere.  If $M$ is an arbitrary module over a ring
$R$, then $\Mod_R(M)$ denotes the category with objects all $R$-submodules of $M$ and as morphisms
all $R$-homomorphisms between these modules.

 Assume that the element $\nabla\in \Dc$ has an adjoint action on a
$\Cb$-algebra $\Dc$, $P \mapsto [\nabla, P]$, which is semisimple, and that the semisimple
decomposition gives a grading $\Dc= \oplus \Dc(n)$, where $P\in \Dc(n)$ if $[\nabla, P]= n P$.  We
make the triangular decomposition
\begin{displaymath}
  \tag{T}
  \Dc = \Dc^- \oplus \Dc^0 \oplus \Dc^+
\end{displaymath}
where $\Dc^{-}= \oplus_{n<0} \Dc(n)$, $\Dc^0= \Dc(0)$, and $\Dc^+ = \oplus_{n>0} \Dc(n)$. Define
also the ring $\Rc = \Dc^0/\Dc^0 \cap (\Dc \Dc^-)$.  

Define the functor
\begin{align*}
  \ell : \Mod_{\Dc}(M) &\to \Mod(\Rc),\\ N &\mapsto \ell (N) =  Hom_{\Dc}(\frac{\Dc}{\Dc \Dc^-},
                                             N)= \{ n\in N\ \vert \ \Dc^-n=0\} 
\end{align*} 
and the map
\begin{align*}
  \delta :  \Mod_{\Rc}(\ell(M)) &\to  \Mod_D(M),\\   V &\to  \delta(V)=\Imo ( \frac{\Dc}{\Dc \Dc^-} \otimes_{\Rc} V\to M) = \Dc V.
\end{align*}
Here $\Dc/\Dc\Dc^-$ is a $(\Dc, \Rc)$-bimodule, so that one gets the adjoint pair of functors
$(\Dc/\Dc\Dc^-\otimes_{\Rc}\cdot , Hom_{\Dc}(\Dc/\Dc\Dc^-, \cdot ))$, while $\delta$ in general does
not give a functor on the category $ \Mod_{\Rc}(\ell (M))$. However, if $M$ is sufficiently nice we
do get a functor.

 In the main part of the paper we will use the more evocative and convenient notation $M^{ann}=\ell(M)$. 

\begin{theorem}\label{main}
  Let $M$ be a semisimple $\Dc$-module which is semisimple over $\nabla$ and satisfying
  $\delta\circ \ell(M) = M$. Then $\ell : \Mod_\Dc(M)\to \Mod_{\Rc}(\ell(M))$ defines an isomorphism
  of categories, with inverse $\delta :\Mod_{\Rc }(\ell (M)) \to \Mod_\Dc(M) $.
\end{theorem}
Note that we actually have an isomorphism of categories in the theorem, not only an equivalence, and
that this isomorphism preserves the subcategories with the same objects, but in which the morphisms
are restricted to being inclusions of submodules.

With the {\it support} \/ of a semi-simple $\Cb [\nabla]$-module is meant the set of non-zero
eigenvalues of $\nabla$.

\begin{lemma}\label{verma-lemma}
  Let $W$ be a simple $\Rc$-module, that is semi-simple as a $\Cb [\nabla]$-module, and which we also regard as a simple module over the ring
  $\Bc= \Dc^0+\Dc^-$ by the projection $\Dc^0 \to \Rc$ and trivial action of $\Dc^-$. Then
 \begin{enumerate} 
 \item The support of $W$ as a $\Cb [\nabla]$-module consists of one element.
  \item $\Dc\otimes_{\Bc} V$ contains a unique maximal submodule.
  \end{enumerate}
\end{lemma}
\begin{proof} (1) is clear since $\Rc$ preserves any eigenspace of $\nabla$. Also 
$\nabla$ acts semi-simply on $ \Dc\otimes_{\Bc}W$ as a derivation by $\nabla(Q\otimes v)=[\nabla,Q]\otimes v+
Q\otimes \nabla v$.
  Since $W$ is simple, the support of any proper submodule of
  $ \Dc\otimes_{\Bc}W$, regarded as $\Cb [\nabla]$-module, is disjoint
  from the support of $W$. The maximal proper submodule is then the sum of all proper submodules.
\end{proof}

\begin{pfof}{\Theorem{\ref{main}}}
  All direct sums below are internal, and by an $\nabla$-isotypical component of $M$ associated to $\lambda$ we intend the  subspace of $M$ consisting of elements $m$ such that $\nabla \cdot m = \lambda
  m$. Let $V$ be a $\Rc$-submodule of $\ell (M)$ and $N$ a submodule of $M$.

  (a) $\delta\circ \ell (N) =  N$: If $N$ is a submodule of
  $M$, by semisimplicity there exists a module $N_1$ such that $M= N\oplus N_1$, so that 
  \begin{displaymath}
N\oplus N_1= M  = \delta \ell (M) = \delta (\ell (N)) \oplus \delta (\ell (N_1))
\end{displaymath}
and hence $N= \delta (\ell (N))$.

(b) $\ell\circ \delta (V) = V$: We note that it follows from the decomposition \thetag{T}, that if $N\subset M$ is
 a simple $\Dc$-module, $\ell(N)$ contains only one isotypical component with respect to the action of $\nabla$.
Assume first that
 $V$ contains only a single $\nabla$-isotypical component, and that 
 $\delta (V) = \oplus_{i\in I} N_i,$ where $N_i$ are
simple $\Dc$-submodules of $M$. Hence 
\begin{displaymath}\tag{*}
  \ell (\delta (V)) = \bigoplus_{i\in I} \ell (N_i).
\end{displaymath}
Since $\delta (V)=V\oplus \Dc^+V$, and $V\subset \ell (\delta (V))$, it is clear that $ \ell (\delta (V))=V\oplus V'$, where $V$ and $V'$ have different $\nabla$-isotypic components. Thus, there is a subset $I'\subset I$ such
that $V= \oplus_{i\in I'} \ell (N_i)$.  Then 
\begin{displaymath}
  \delta(V) = \bigoplus_{i\in I'} \delta( \ell (N_i) )= \bigoplus_{i\in I'} N_i = \bigoplus _{i\in I} N_i,
\end{displaymath}
where the second equality follows from (a). Therefore $I=I'$ and so $\ell\circ \delta (V)$ on the right side of \thetag{*} equals $V$. 

Assume then that  $V=V_1\oplus V_2$, where $V_1$ and $V_2$ have no common $\nabla$-isotypical component, and satisfy that $\ell\circ \delta (V_i) = V_i,\ i=1,2$.   
Then
$\ell (\delta (V_1)\cap \delta (V_2)) \subset V_1 \cap V_2 =0$, hence by (a)
$\delta (V_1)\cap \delta (V_2)=0$, so that 
\begin{displaymath} \delta(V_1 \oplus V_2) = \delta
  (V_1)\oplus \delta(V_2), 
\end{displaymath} 
and, by assumption,
$$
\ell (\delta (V))=\ell(\delta(V_1)) \oplus \ell(\delta(V_2)) = V_1\oplus V_2.$$
Since any $V$ may be decomposed as a $\Dc^0$-module into isotypic components for $\nabla$, it follows by induction that $\ell\circ \delta (V) = V$.

(c) $\Hom_{\Dc} (N_1, N_2) = \Hom_{\Rc} (\ell (N_1), \ell (N_2))$: Since $\ell$ is additive and the
category $\Mod_\Dc(M)$ is semisimple, it suffices to prove this when $N_1$ and $N_2$ are simple. If
$V_1\subset \ell(N_1)$, $V_1\neq 0$, then $\delta (V_1) = N_1$; hence by (b) $V_1= \ell(N_1)$; hence
$\ell(N_1)$ is simple and for the same reason $\ell (N_2)$ is also simple. It is obvious that an
isomorphism $\phi : N_1 \to N_2$ induces an isomorphism $\ell(N_1)\to \ell (N_2)$.  Conversely, let
$\psi : \ell(N_1)\to \ell (N_2) $ be a non-zero homomorphism, hence it is an isomorphism.  There is
a canonical inclusion homomorphism of $\Dc^0$-modules $f:\ell(N_2) \to N_2$, so that we get a map of
$\Dc^0$-modules $f\circ \psi : \ell(N_1)\to N_2$. Hence we get a non-zero homomorphism of
$\Dc$-modules $h: \Dc\otimes_{\Bc}\ell(N_1)\to N_2 $, which is surjective since $N_2$ is simple. We
moreover have a surjective map $ \Dc\otimes_{\Bc}\ell(N_1)\to N_1$.  Since $\ell(N_1)$ is simple, by
\Lemma{\ref{verma-lemma}} $ \Dc\otimes_{\Bc}\ell(N_1)$ has a unique maximal proper
submodule. Therefore we get a unique isomorphism $N_1 \to N_2$ that extends $\psi$.

(d) $\Hom_{\Rc}(V_1, V_2) = \Hom_{\Dc}(\delta(V_1), \delta (V_2))$:
Putting $N_1= \delta(V_i)$ we have by (b) that $\ell(N_i)= V_i$,
$i=1,2$. Hence by (c)
\begin{displaymath}
\Hom_{\Rc}(V_1, V_2) = \Hom_{\Rc}(\ell(N_1), \ell(N_2)) =
\Hom_{\Dc}(N_1, N_2) = \Hom_{\Dc}(\delta(V_1), \delta(V_2)).
\end{displaymath}
\end{pfof}

\subsection{Equivalence for invariant rings}\label{inv-rings}
\Theorem{\ref{main}} applies immediately to algebras $\Dc=\Dc_B^G$ of invariant differential
operators, where we use the notation of \Section{\ref{section:notation}}.  The semi-simple adjoint
action of the Euler operator $\nabla$ induces a $\mathbb Z$-grading, we have the decomposition
\thetag{T}, and by \Proposition{\ref{thm:montgomery}} $B$ is a semi-simple $\Dc$-module. If
$N\subset B$ is a simple $\Dc$-submodule, the vector space $N^a$ of lowest degree homogeneous
elements in $N$ will be annihilated by $\Dc^-$, hence $N^a\subset B^{ann}$, so that
$N=\Dc N^a\subset \Dc B^{ann}$.  Since this is true for any simple submodule of the semi-simple
module $B$, we have $B=\Dc \cdot B^{ann}$.  Hence the conditions of \Theorem{\ref{main}} are
obtained for $\Dc$ and $M=B$.  We can therefore immediately conclude most of the following basic
result:
\begin{corollary}
\label{cor:d0eqconcrete}
Suppose that $G$ is a finite group acting on $B=\So(V)$ and $\Dc_B$.  Put $\Dc=\Dc_B^G$ and
$\Rc = \Dc^0/\Dc^0 \cap (\Dc \Dc^-)$.
 \begin{enumerate}
\item  $B^{ann}$ is a finite-dimensional semisimple $\Rc$-module.
\item There is an isomorphism of categories between the category of
  $\Dc$-submodules of $B$ and the category of $\Rc$-submodules of
  $B^{ann}$. The isomorphism is $N\mapsto N^{ann}$, where $N$ is a
  submodule of $B$, and its inverse is $V \mapsto \Dc V$, where $ V$
  is an $\Rc$-submodule of $B^{ann}$.
 \item Simple $\Dc$-submodules of $B$ correspond to simple $\Rc$-submodules of $B^{ann}$.
 \item Each simple $\Rc$-submodule of $B^{ann}$ is concentrated in a single degree.
\end{enumerate}
\end{corollary} 
\begin{proof}
  There remains to prove $\dim_{\Cb} B^{ann} < \infty$. This follows
  from \Proposition{\ref{thm:montgomery}} and \thetag{A}, together
  with the following facts: $\dim V_\chi < \infty$,
  $\Ann_{\Dc^-}(B_\chi)$ is concentrated in one degree, and each
  homogeneous degree of $B$ is of finite dimension.
\end{proof}

\begin{remark} \label{harmonics} The ring $\So(V^*)$ is isomorphic to the subring of constant
  differential operators in $\Dc$ and $\Hc= \Hom_{\So(V^*)}(\So(V^*)/\So(V^*)^G_+, B) $ is the space
  of harmonic elements in $B$, so that $B^{ann}\subset \Hc $. Since $\Hc\cong B/\mf_A B$, this gives
  another argument for $\dim_{\Cb} B^{ann}< \infty$. The space $\Hc$ is isomorphic to the regular
  $G$-representation if and only if $G$ is a complex reflexion group
  \cite{steinberg:differential}. In this case the $A$-modules $N_\chi$ in
  \Proposition{\ref{thm:montgomery}} are free of rank $\dim_{\Cb} V_\chi$.
\end{remark}

To fix ideas we give a non-trivial example.
\begin{example} Let the symmetric group $G=S_3$ act on $ B = \Cb [x_1,x_2,x_3]$ by permuting the
  variables; put $A=B^G$ and $\Dc= \Dc_B^G$. Let $\alpha_{ij}=x_i-x_j, \ i\neq j$ be the equations of irreducible
  reflecting hyperplanes. Then
  $B^{ann}= \Cb 1+ \Cb \alpha_{13}+ \Cb\alpha_{23}+ \Cb \alpha_{12} \alpha_{13} \alpha_{23}$ is a
  four-dimensional vector space, and the simple $\Rc$-modules are the one-dimensional vector spaces
  $ N_0=\Cb$, $N_1= \Cb \alpha_{12}\alpha_{13}\alpha_{23}$ and $N_p= \Cb p$ where
  $p\in \langle \alpha_{13},\alpha_{23}\rangle$. Corresponding representants for the three classes
  of simple $\Dc$-modules are $\Dc \cdot N_0=A$,
  $\Dc \cdot N_1=A\alpha_{12}\alpha_{13}\alpha_{23}$ and, selecting $p= \alpha_{12}$,
  $\Dc\cdot N_{x_1-x_2}=A(x_1-x_2)\oplus A(x_1^2-x_2^2)$. A complete description of the
  $\Rc$-module structure of $B^{ann}$ for the general symmetric group $S_n$ is given
  in \Section{\ref{branchrule}}.
\end{example}

\subsection{Computation of $\Dc^-$, $\Dc^0$ and $\Rc$} \label{section:d0} \subsubsection{General
  procedure using basic invariants}\label{gen-proc} Assume that $\{f_i\}$, $\{g_i\}$ are homogeneous generators of
$\So(V)^G$ and $\So(V^*)^G$, respectively, where one observes that $\{g_i\}\subset  \Dc^-$.  Let $\afr$ be a
graded Lie subalgebra of $\Dc = \Dc(V)^G$ which contains the Lie algebra $\text{Lie} <f_i, g_j>$
that is generated by the set $\{f_i\}\cup \{g_i\}$.

Letting $\Uo(\afr)$ be the enveloping algebra of $\afr$ we have a canonical
homorphism 
\begin{displaymath} \jo : \Uo (\af) \to \Dc.
\end{displaymath} 
\begin{proposition}\label{gen-R} \label{prop:pbw}
  \begin{enumerate}
 \item The homomorphism $\jo$ is surjective.
 \item Let $\{r_k\}$ be a homogeneous basis of $\af$ and
   $\afr^- = \sum_{\deg (r_k) < 0} \Cb r_k $ be the subalgebra of
   elements of negative degree. Then
   $ \afr^- \subset \Dc^- \subset \Dc \afr^- $.
  \item Let $\af^0$ be the subalgebra of degree $0$ in $\af$. Then we
  have a surjective homomorphism
  \begin{displaymath}
    \Uo(\af^0) \to \Rc
  \end{displaymath}
  In particular, if $\af^0$ is commutative, then $ \Rc$ is commutative.
  \end{enumerate}
\end{proposition}  It follows from (3) that $M^{ann} = \Ann_{\afr^-}(M)$ and from (4)
  that the $\Rc$-module structure of $M^{ann}$ come from its structure as $\afr^0$-module.  
\begin{proof}
  (1): This follows from the famous theorem of Levasseur and Stafford
  \cite{levasseur-stafford:invariantdiff}, stating that the sets $\{f_i\}$, $\{g_i\}$ together generate $\Dc$.  

  (2-3): Provide the homogeneous basis $\{r_k\}$ with a total ordering that is compatible with
  the degrees in the sense that $\deg (r_i)\leq \deg (r_j)$, when $i \geq j$. It follows from the
  Poincar{\'e}-Birkhoff-Witt theorem, using the homomorphism above, that any element $P\in \Dc$ can be
  expressed (non-uniquely) in the form
\begin{displaymath}
  P= \sum \alpha_{i_1 \ldots i_s}  r_{i_1} \cdots r_{i_s}
\end{displaymath}
where $i_1 \geq \cdots \geq i_r$.  Hence the factors in each term $ r_{i_1} \cdots r_{i_s}$ have
descending degree $\deg r_1 \geq \cdots \geq \deg r_{i_s}$.  If $P\in \Dc^-$ then the last factor in
each term has $\deg(r_{i_s})<0$ and so $P\in \Dc \afr^-$. This gives (2).  Furthermore, when
$P\in \Dc^0$ we can write
\begin{displaymath}
\label{eq:pbw}
  P =  \sum \alpha_{i_1\ldots i_s}  r_{i_1} \cdots r_{i_s} \mod \Dc^0\cap \Dc\Dc^-,
\end{displaymath}
where $\deg (r_{i_1})= \cdots = \deg (r_{i_s}) =0$ and $ \alpha_{i_1\ldots i_s}\in \Cb$. This gives
(3).
\end{proof}

\subsubsection{$\Rc$ for $G(m,e,n)$}\label{imprimitive}
As an example we will consider the irreducible imprimitive complex reflection groups $G=G(m,e,n)$,
where $e$ and $m$ are positive integers such that $e \ \vert \ m$, and determine generators of $\Rc$
when $e=1$. Let $V$ be a complex vector space of dimension $n$. Then
\begin{displaymath}
G= A(m,e,n)\rtimes S_n \subset \Glo(V)
\end{displaymath}
where $S_n$ is realized as permutation matrices and $A(m,e,n)$ as diagonal matrices whose entries
belong to $\mu_{m}$, the group of $m$-roots of unity, such that their determinant belongs to
$\mu_d\subset \mu_{m}$, where $d= m/e$, and in the semi-direct product $S_n$ acts on $A(m,e,n)$ by
permutation.  This means that $G$ can be realized as permutation matrices with entries in $\mu_{m}$,
such that the product of the non-zero entries belongs to $\mu_d\subset \mu_{m}$; see \cite{broue}.
Here $G(1,1,n)= S_n$ is of type $A_{n-1}$, $G(2,1,n)$ is of type $B_n=C_n$; the dihedral group
$G(e,e,2)= I_2(e) $, where $I_2(6)$ is of type $G_2$, is treated in more detail in (\ref{dihedral});
$G(2,2,n)$ is of type type $D_n$, and $G(de,e,1)= C_d$ is a cyclic group.

Put $\bar A = A(m,1,n)$, $A= A(m,e,n)$, and $\bar G = \bar A \rtimes S_n= G(m,1,n)$ so that
$A\subset \bar A$ and $G\subset \bar G $. Define the $S_n$-invariant elements
$\Theta = (\prod_{i=1}^n x_i)^d$, $\Psi = (\prod_{i=1}^n \partial_i)^d$, and  $h_i(x)=
\sum_{j=1}^n x_j^i$.

\begin{lemma}\label{power-inv}
  \begin{enumerate}
\item
  \begin{align*}
B^{\bar G} &= \Cb [f_1, \ldots , f_n]\\ &\subset  B^G= \Cb [f_1, \ldots , f_{n-1}, \Theta]
\end{align*}
where 
\begin{displaymath}
  f_i = h_i (x_1^{m}, \dots , x_n^{m})= \sum_{j=1}^n x_j^{mi}, \quad i=1, \dots , n
\end{displaymath}
\item
  \begin{align*}
\bar \Dc_n &:= \Dc_B^{\bar G}=  \Cb \langle f_i(x), f_i(\partial)\rangle \\ &\subset  \Dc_B^{G}=  \Cb \langle f_i(x), f_i(\partial), \Psi, \Theta\rangle
\end{align*}
  \end{enumerate}
\end{lemma}
\begin{proof}
(1):  This is of course well known, but let us at least sketch the argument.  We have (as detailed
  below)
  \begin{eqnarray*}
   B^{A(m,e,n)\rtimes S_n} &=& (B^{A(m,e,n)})^{S_n} = (\Cb[x_1^{m}, \dots , x_n^{m},(x_1\cdots
         x_n)^d ])^{S_n}\\ &=& (\Cb [(x_1\cdots
                               x_n)^d ] [x_1^{m}, \dots , x_n^{m}])^{S_n}= \Cb [(x_1\cdots
                               x_n)^d ][f_1, \dots , f_n]
\end{eqnarray*}
The second equality can be seen by first noting that if a polynomial is $A(m,e,n)$-invariant, then
each of its monomial terms is invariant, and these are exactly given by powers of the monomials
$ x_1^{m}, \dots , x_n^{m},(x_1\cdots x_n)^d$. The $n$ monomials $x_i^m$ are algebraically
independent, while the last is $S_n$-invariant and algebraically dependent on the other ones.
Therefore the second equality on the second line follows from the well-known fact that
$\{f_i\}_{i=1}^n$ is an algebraically independent set of generators of
$\Cb[x^{m}_1, \dots, x^{m}_n]^{S_n}$.  When $e> 1$, then
$f_n \in \Cb [f_1, \dots , f_{n-1},\Theta]$, while if $e=1$ (so that $m=d$) , then
$\Theta \in \Cb[f_1, \dots , f_n]$, so that $B^{\bar G}= \Cb[f_1, \dots , f_n ]$.

(2): This follows from the theorem of Levasseur and Stafford \cite{levasseur-stafford:invariantdiff}.
\end{proof}

\Lemma{\ref{power-inv}} means that we have good control of the invariants $\{f_i\}$ for the ring
$B^{G(m,e,n)} $ which are needed in \Proposition{\ref{gen-R}}. However, it is only in the case $e=1$
that we obtain a really useful description of the Lie subalgebra
$\bar \af = Lie <f_i(x), f_i(\partial)>$ of $\bar \Dc_n$ (or more precisely of a Lie algebra
containing $\bar \af $), using the basic invariants $f_i(x)= h_i(x_1^{m}, \ldots , x_n^m )$ for
$B=\So(V)$ and $f_i(\partial)$ for its isomorphic ring $\So(V^*)$.



\begin{prop}\label{symm-lie} 
  \begin{enumerate}
  \item The Lie algebra $\bar \af$ is contained in a Lie algebra $\bar \af'$ with the basis
    \begin{displaymath}
      \left\{\sum_{i=1}^n  x_i^{k}  \partial_i^l\right \}_{k \geq 0, l\geq 0, m \vert l-k}.
    \end{displaymath}
    If $m=1$, then $\bar \af = \bar \af'$.
  \item A basis of $(\bar \afr')^-$ is provided by the elements in (1) with $0 \leq k < l$. If
    $z\in B^{ann}$, then its degree $\deg (z) < nm$. In particular, if $M$ is a
    $ \bar \Dc_n$-submodule of $B$, then 
      \begin{displaymath}
\Ann_{\bar \Dc^-_n}(M) =  \{z \in M \ \vert \  \sum_{i=1 }^n  x_i^{k}  \partial_i^l \cdot z =0, \text{when }
0\leq k < l \leq nm-1, \ m\ \vert \  l-k  \}.
\end{displaymath}
 \item   We have
  \begin{displaymath}
    (\bar \af')^0   = \Cb[\nabla_1, \dots , \nabla_n]^{S_n},
\end{displaymath}
and therefore $\bar \af^0$ is commutative.
  \end{enumerate}
\end{prop} 
Let us agree to call $\sum_{i=1}^n x_i^{k} \partial_i^l$, $m\vert (l-k)$, a {\it power differential
  operator}.  Of particular interest is of course the case $m=1$, so that $\bar G= S_n$ and
$\bar \Dc_n$ is the ring of symmetric differential operators, and there are then no restrictions on
$l-k$.

\begin{lemma}\label{cyclic-do} Let $\Dc(\Cb)^{Lie}= \Cb\langle x,\partial\rangle ^{Lie}$ be the Weyl algebra in $1$
  variable, considered as a Lie algebra. Let the cyclic group $C_m$ act on $x$ by a primitive $m$th
  root of unity, and thus inducing an action on $\Dc(\Cb)^{Lie}$. Then the invariant algebra
  $(\Dc(\Cb)^{Lie})^{C_m} = \Cb \langle x^m, \partial^m\rangle ^{Lie}$. If $m=1$, then this Lie
  algebra is generated by the set $\{x^{k}, \partial^{l}\}_{0 \leq k \leq 3, 0 \leq l \leq 3}$.
\end{lemma}

\begin{remark}\label{rem-sn}
  \begin{enumerate}
  \item Notice (in the proof) that if $n> 2$ then $\bar \af'$ is isomorphic to the Lie algebra
    $(\Dc(\Cb)^{C_m})^{Lie}$ in \Lemma{\ref{cyclic-do}}, and
    $(\bar \af')^0\cong \Cb [\nabla]^{Lie}\subset (\Dc(\Cb)^{Lie})^{C_m}$, where
    $\nabla = x \partial_x$. The Lie subalgebra $(\bar \af')^-$ is not finitely generated.
  \item When $n=2$ and $m=1$ then the Lie algebra
    $\bar \af = \Cb \cdot 1 + \Cb \nabla + \sum_{i=1}^2 (\Cb f_i(x)+ \Cb f_i(\partial))$ is
    finite-dimensional and $\bar \af^0 = \Cb \cdot 1 + \Cb \nabla $.
  \item One may ask when the Lie algebras $\bar \af$ are isomorphic for different choices of basic
    sets of invariants.  For instance, when $n=2$ and $m=1$ using the basic invariants
    $e_1= x_1+x_2 , e_2= x_1x_2 $ we get
    $\af^{(1)} = \Cb \cdot 1 + \Cb \nabla + \sum_{i=1}^2 (\Cb e_i(x)+ \Cb e_i(\partial))$, so that
    $\af^{(1)} \neq \af$ but still $\af^{(1)} \cong \af$.
  \end{enumerate}
\end{remark}
\begin{proof} The equality $\Dc(\Cb)^{C_m} = \Cb \langle x^m, \partial^m\rangle $ follows from the
  theorem of Levasseur and Stafford \cite{levasseur-stafford:invariantdiff} (already when $m=3$ it
  is a nontrivial fact that $x\partial \in \Cb \langle x^m, \partial^m\rangle $), so that
  $(\Dc(\Cb)^{Lie})^{C_m} = \Cb \langle x^m, \partial^m\rangle ^{Lie}$.  It is elementary to see,
  however, that the set $\{ x^{k}\partial^{l}\}_{m \vert (l-k)}$ is a basis of
  $(\Dc(\Cb)^{Lie})^{C_m}$.  Assume now that $m=1$ and let $\of$ be Lie algebra that is generated by
  the elements $\{x^{r},\partial^{s}\}_{0 \leq r \leq 3, 0 \leq s \leq 3}$.  Since
  $[\partial^2,x^3]=6 x^2\partial + 6x $ and $[\partial^3, x^2 ]= 6x\partial^2+ 6x$ it follows that
  $E_x= x^2\partial, E_\partial=x\partial^2\in \of$. Now
  \begin{displaymath}
  [x^{k_1}\partial^{l_1}, x^{k_2}\partial^{l_2}]= (l_1k_2-l_2k_1) x^{k_1+k_2-1}\partial^{l_1+l_2-1} + (l.o.) 
\end{displaymath}
where "l.o.'' signifies a linear combination of terms $x^r\partial^s$ where
$r < k_1+k_2-1, s< l_1+l_2-1 $.  In particular,
  \begin{displaymath}
    [E_x, x^k\partial^l] = (k-2l) x^{k+1}\partial^l + (l.o.) , \quad      [E_\partial,
    x^k\partial^l] = (2k- l) x^{k}\partial^{l+1} + (l.o.). 
  \end{displaymath}
  A straightforward  induction in $k$ and $l$ now shows that $\of = \Dc(\Cb)^{Lie}$.

\end{proof}

\begin{pfof}{\Proposition{\ref{symm-lie}}}
  (1): This is essentially a 1-variable assertion, based first on the fact that
  \begin{displaymath}
  [x^k_i\partial^{l}_i, x^{k_1}_i\partial^{l_1}_i]= \sum_{j=1}^{r} c_j x^{k+k_1-j}_i\partial^{l+l_1-j}_i, 
\end{displaymath}
where $r = \max (\min(l,k_1), \min(k,l_1))$ (unless the Lie bracket is 0), where of course the
coefficients $c_j$ do not depend on $i$, and secondly that variables do not mix in Lie brackets
since $[x^k_i\partial^{l}_i, x^{k_1}_j\partial^{l_1}_j]=0$ when $j\neq i$.  If $m\vert (l-k)$ and
$m \vert (l_1- k_1)$, then $m \ |\ (k+k_1 -j - (l+l_1-j))$. This implies that the bracket of two
power differential operator is a linear combination of power differential operators, so that the
vector space $\bar \af'$ that is spanned by such differential operators forms a Lie algebra, and
clearly $\bar \af \subset \bar \af'$. Assume that $m=1$.  By \Lemma{\ref{cyclic-do}} the set
$\{x_i^{k}, \partial^{l}_i\}_{0 \leq k \leq 3, 0 \leq l \leq 3}$ generates $\Dc(\Cb)^{Lie}$, which
implies that the corresponding set of power differential operators
$\{\sum_{i=1}^nx_i^{k}, \sum_{i=1}^n\partial^{l}_i\}_{0 \leq k \leq 3, 0 \leq l \leq 3}$ generates
$\bar \af$, i.e.  $\bar \af' = \bar \af$.

(2): It is evident by definition that $(\bar \af')^-$ has the given basis of power differential
operators $p_{k,l} = \sum_{i=1 }^n x_i^{k} \partial_i^l$, $l > k$, $m\ \vert \ (l-k)$.  Since
$(\bar \af')^- \subset \bar \Dc_n^-$ it is also evident that if $z \in \Ann_{\bar \Dc_n^-}(M)$, then
$p_{k,l} \cdot z =0$, when $0 \leq k < l \leq nm$, $m \ \vert \ (l-k)$. Assume now the converse,
that $p_{k,l} \cdot z =0$ for such $l$ and $k$, so that in particular $p_{0,lm} \cdot z =0$ for
$l=0, \dots , {n} $. We have
 \begin{displaymath}
Q(\partial^m)= \prod_{i=1}^n(\partial^m - \partial^m_i)=\partial^{nm} + \sum_{i=1}^{n} (-1)^ie_i(\partial^m_1, \dots
, \partial^m_n) \partial^{(n-i)m},
\end{displaymath}
where the elementary symmetric polynomials
\begin{displaymath}
  e_i=e_i(\partial^m_1, \dots, \partial^m_n) \in \Cb [\partial^m_1,\ldots,  \partial^m_n]_+^{S_n} = \Cb [p_{0,m}, \ldots , p_{0,{nm}}]_+,
\end{displaymath}
so that $e_i\cdot z =0$. Since $Q(\partial^m_i)=0$, it follows that $\partial_i^{nm} \cdot z =0$ and
hence $\deg_i (z)< nm$, $i=1, \ldots, n$.  Therefore $p_{k,l}\cdot z =0 $ also when $l \geq nm$.
Since $\bar \Dc_n^- \subset \bar \Dc_n \cdot (\bar \af')^-$, by \Proposition{\ref{gen-R}} (2), it
follows that $\bar \Dc_n^- \cdot z =0 $.

(3): Clearly $f_j(\nabla_1, \dots , \nabla_n) $ is a power sum of degree $0$ so it belongs to
$(\bar \af')^0$; hence $\Cb[\nabla_1, \dots , \nabla_n]^{S_n}\subset (\bar \af')^0$.  Conversely,
since $x_i^k\partial^k_{i} = p_k(\nabla_i)$ (see \Lemma{\ref{lemma:nabla}}) it is also clear that
$(\bar \af')^0 \subset \Cb[\nabla_1, \dots , \nabla_n]^{S_n} $.
\end{pfof}

In general $\af$ is an extension of $\bar \af$ by the elements $\Theta$ and $\Psi$, similarly to
\Lemma{\ref{power-inv}},  and these elements will mix the variables, making it considerably more
difficult to describe bases of $\af$, $\af^-$ and $\af^0$.  The cases $n=2$ and
$e \geq 1$ are studied in  \Section {\ref{cycl-dihedral}}.

\subsubsection{Using $\gl(V)$ to determine generators of $\Rc$}

There is another way to think of (3) in \Lemma{\ref{symm-lie}}. 
We have inclusions
\begin{displaymath}
  \hf \subset \gl(V) \subset \Dc^0(V)
\end{displaymath}
where $\hf$ is a Cartan algebra in the general  Lie algebra $\gl(V)$, and we have a surjective map 
$l: \Uo(\gl(V))\to \Dc^0(V)$, where $\Uo(\gl(V))$   is the enveloping algebra of $\gl(V)$.
Since $G$ is finite the induced (and same noted) map
\begin{displaymath}\label{diff-map}
 l: \Uo(\gl(V))^G \twoheadrightarrow \Dc^0
\end{displaymath}
is again surjective.  The maximal subgroup of $\Glo(V)$ that preserves the Cartan algebra $\hf$ is
of the form $T \rtimes S_n$, where $S_n$ is the symmetric group and the torus $T$, is the maximal
subgroup that leaves $\hf$ invariant, where moreover $T$ acts trivially on $\hf$. Thus, if $G\subset
T \rtimes S_n $ and $\bar G $ is the image of $G$ in $S_n$, then $\bar G$ preserves $\hf$,
and we have the commutative subring
\begin{displaymath}
   l( \So(\hf)^{\bar G}) \subset \Dc^0.
\end{displaymath}
This is in particular true when $G=G(m, e,n)$, where $A(m,e,n)\subset T$, so that we get the subring
$ l(\So(\hf)^{S_n})$.  \Lemma{\ref{symm-lie}} therefore implies
\begin{proposition}\label{gen-sym-R}
  If $G=G(m,1,n)$, then
\begin{displaymath}
l(\So(\hf)^{S_n})  \mod \Dc^0\cap (\Dc\cdot \Dc^-)  =  \Rc.
\end{displaymath}
\end{proposition}

We may also use $l$ and invariant theory of commutative rings, to find algebra generators of
$ \Dc^0$. (This will be the method used in \Section{\ref{cycl-dihedral}}, for the cyclic and
dihedral groups.)

The natural order filtration $\{\Dc_n(V)\}$ of $\Dc(V)$ is $G$-invariant,
$G\cdot \Dc_n(V)\subset \Dc_n(V)$ and therefore induces a filtration $\Dc_n^0=\Dc_n(V)^G \cap \Dc^0$
of the subring $\Dc^0\subset \Dc \subset \Dc(V)$.  Similarly, the enveloping algebra $\Uo((\gl(V))$
is also provided with a natural filtration $\{\Uo_n(\gl(V))\}$ such that
$G\cdot \Uo_n(\gl(V)) \subset \Uo_n(\gl(V))$, and thus induces a filtration $\{\Uo_n(\gl(V))^G\}$ of
$\Uo(\gl(V))^G$.  Put 
$$\gr^\bullet (\Dc^0)= \bigoplus_{n\geq -1} \frac{\Dc_{n+1}^0}{\Dc_n^0} $$ (where
$\gr^0(\Dc^0) = \Sc(V)^G$ ). By the   Poincar{\'e}-Birkhoff-Witt theorem
$$\bigoplus_{n\geq -1}  \frac{\Uo_{n+1}(\gl(V))}{\Uo_n(\gl(V))} = \So^\bullet (\gl (V)),$$ and therefore
$$ \bigoplus_{n\geq -1} \frac{\Uo_{n+1} (\gl(V))^G}{\Uo_{n} (\gl(V))^G} = \So^\bullet (\gl (V))^G,$$since
taking $G$-invariants is an exact functor.  By the same reason,
we have that $$l(\Uo_n(\gl(V))^G) = \Dc_n^0,$$so there is a surjective homomorphism 
  \begin{displaymath}
    l^{gr}:    \So^\bullet (\gl (V))^G\to \gr^\bullet (\Dc^0).
  \end{displaymath}
 Since a homogeneous set of generators of $\gr^\bullet (\Dc^0)$
  can be lifted to generators of $\Dc^0$, we conclude that to get generators of $\Dc^0$ it suffices
  to compute generators of $ \So^\bullet (\gl (V))^G= \So^\bullet (V \otimes_{\Cb} V^*)^G $, where
  $G$ acts diagonally on $V\otimes_{\Cb} V^*$.  We summarize this in a lemma:
\begin{lemma}\label{lemma:gr} Let $\{\bar a_i\}_{i\in I}$   be a set of 
  homogeneous elements in $\So(V\otimes_{\Cb} V^*)^G$ that generates $\So(V\otimes_{\Cb} V^*)^G$, and
  let $\{ a_i \}_i \in I$ be a subset of $\Uo((\gl(V))$ such that $a_i$ represents $\bar a_i$. Then
  $\{l(a_i)\}_{i\in I}$ is a generating subset of $\Dc^0$.
\end{lemma}
\begin{remark}
  The map $l^{gr}$ is in fact a homomorphism of Poisson algebras. Using the Poisson product one can
  sometimes, for example when $G$ is a Weyl group with no factors of type $E_m$, prove that
  generators of the subrings $\So^\bullet (V)^G$ and $\So^\bullet (V^*)^G$ together generate the
  Poisson algebra $\So^\bullet (V \oplus V^*)^G$, which then can be lifted to generators of
  $\Dc(V)^G$; see \cite{wallach:invariantdiff, levasseur-stafford:invariantdiff}.  For
  $G=A(e,e,1)\subset \Glo(\Cb^1)$ one does not get generators of the whole Poisson algebra in this
  way, but even so lifts of the generators of $\So^\bullet (V)^G$ and $\So^\bullet (V^*)^G$ do give
  generators of $\Dc(V)^G$; see \cite[p.371]{levasseur-stafford:invariantdiff}.
\end{remark}

\subsection{Gelfand models and $\Rc$}\label{gelfand-section}
The space $B^{ann}$ has been considered by other authors, under the name of the polynomial model, in
the context of finding {\it Gelfand models}\/ of a finite group $G$ (see e.g. \cite{generalized,
  polynomial,gelfand.garge}). Such a model is defined to be a $G$-representation that is a direct
sum of a representative of each isotypical class in $\hat G$. Now, by \Proposition{\ref{thm:montgomery}}, we have
\begin{displaymath}
  B^{ann}\cong \bigoplus_{\chi \in \hat G } V_\chi \otimes N_\chi^{ann},
\end{displaymath}
so $B^{ann}$ is a Gelfand model if and only if, for all $\chi\in \hat G$, $N_\chi^{ann}$ is a
1-dimensional complex vector space. We note the relation to fake degree \cite[5.3.3]{geck}, which is
the   Poincar{\'e} polynomial $P_\chi(t)=\sum \dim_{\Cb} \bar N_\chi(i) t ^i$ of the graded vector space
$\bar N_\chi:=N_\chi/\mf_A N_\chi =\oplus_i \bar N_\chi(i)$. In terms of the fake degree, $B^{ann}$ is a Gelfand model
if and only if, the least non-zero coefficient of all $P_\chi(t)$ is $1$.  The fake degree has been
calculated for Coxeter groups; this is used in \cite{gelfand.garge} to give a uniform proof of the
fact that $B^{ann}$ is a Gelfand model when $G$ is a finite Coxeter group not of type
$D_{2n},\ n\geq 2$, $E_7$ or $E_8$.

 We do know by the correspondence above in \Corollary{\ref{cor:d0eqconcrete}} that $N_\chi^{ann}$ is
 a simple $\Rc$-module. Since the only simple modules over a commutative algebra over
 $\Cb$ are 1-dimensional, we have the following immediate result, that gives a simple proof of
 a main result in \cite{generalized}.

\begin{theorem}
\label{thm:rcomm}If $\Rc$ is commutative, then
 $B^{ann}$ is a Gelfand model of $G$.  In particular this is true for $G=G(m,1,n)$.
\end{theorem} 
  \begin{remark}
    The quotient $\Rc \to\bar \Rc = \Dc^0/(\Ann_{\Dc^0}(B^{ann}))= \End_{G}(B^{ann})$, implying that
    $B^{ann}$ is a Gelfand model if and only if $\bar \Rc $ is commutative. One may ask whether the same connection between
    commutativity and the fact that $B^{ann} $ is a Gelfand model holds for $\Rc$, as for $\bar\Rc$.
    When $G= G(e,e,2)$, a dihedral group, we prove in \Proposition{\ref{prop:dihedral1}} that $\Rc$
    is commutative and hence $B^{ann}$ is a Gelfand model, but on the other hand we will see in
    \Proposition{\ref{prop:cyclic3}} that $\Rc$ is not commutative for the action of a cyclic group
    on $\Cb^2$, and that $B^{ann}$ is then not a Gelfand model.
\end{remark}

\section{Macdonald-Lusztig-Spaltenstein restriction for D-modules}
In (\ref{mls-section}) we present the MLS-restriction functor, which generalises and clarifies a
construction in group theory. This is applied to the case of generalized symmetric groups in
(\ref{gen-symm-modules}), after a  discussion in (\ref{section:partition}) of the polynomials $B$
considered as a module over the ring $R_n=\Cb [x_1\partial_1, \ldots , x_n\partial_n]^{S_n} $.
\subsection{MLS-restriction}\label{mls-section}
Suppose we have an inclusion of graded algebras $ \Dc_2\subset \Dc_1$ of the type
in \Section{\ref{abstract-eq}}, and a $\Dc_1$-module $M$ which is semisimple both as $\Dc_1$- and
$\Dc_2$-module, and moreover that $\Ann_{\Dc_2^{-}}(M)$ generates $M$ over $\Dc_2$. Denote the
category of $\Dc_i$-submodules of $M$ by $\Mod_{\Dc_i}(M)$, $i=1,2$, and similarly the equivalent
categories of $\Dc^0_i$-submodules of $M_i^{ann}:=\Ann_{\Dc^-_i}(M)$ by $\Mod_{\Dc^0_i}(M_i^{ann})$.
Since $ \Dc_2^-\subset \Dc_1^-$ we have
 $$M_1^{ann}\subset  M_2^{ann}$$
 and since also $ \Dc_2^0\subset \Dc_1^0$, there is a restriction
 functor $$\Mod_{\Dc^0_1}(M_1^{ann})\to \Mod_{\Dc^0_2}(M_2^{ann})$$ that takes $V\subset M_1^{ann}$
 to $V\subset M_2^{ann}$. This functor corresponds by the category equivalence in \Corollary{\ref{cor:d0eqconcrete}} to the functor
 $$ \Jo_+: \Mod_{\Dc_1}(M)\to \Mod_{\Dc_2}(M), \quad N\mapsto \Dc_2 \cdot N_1^{ann}.$$
\begin{theorem} \label{prop:MLS} Let $N$ be a $\Dc_1$-submodule of $M$.
\begin{enumerate} 
\item If the restriction of $N_1^{ann}$ to a $\Dc_2^0$-module is simple, then $\Jo_+(N)$ is a simple
  $\Dc_2$-module.  This holds for example if $N_1^{ann}$ is a 1-dimensional complex vector space.
\item  If $N_1^{ann}$ is regarded as a $\Dc_2^0$-module by restriction, then
$$ N_1^{ann}=  (\Jo_+(N))_2^{ann}.$$
\end{enumerate}
 \end{theorem} 
Notice that the assumption in (1) implies, by
\Theorem{\ref{main}}, that $N$ is simple.
\begin{proof} 
  (1): Since $N_1^{ann}$ is a simple $\Dc_2^0$-module, it follows by
  \Theorem{\ref{main}} that $ \Dc_2 \cdot N_1^{ann}$ is a simple
  $\Dc_2$-module.

  (2): By definition, $\Jo_+$ is just restriction from $\Dc_1^0$ to
  $\Dc_2^0$ on submodules of $M^{ann}$, so that using the category
  equivalence of \Theorem{\ref{main}} we get
  $(\Jo_+(N))_2^{ann} =(\Dc_2 \cdot N_1^{ann})_2^{ann} =N_1^{ann}$.
 \end{proof}

We apply the above construction to the ring $B=\So(V)$, where $V$ is a representation of a finite
group $G$, and $H\subset G$ is a subgroup, so that $B$ is both a $G$- and $H$-representation. Set
$\Dc_B=\Dc(V)$ and
\begin{displaymath}
  \Dc_2=\Dc_B^G\subset \Dc_1=\Dc_B^H
\end{displaymath}
Letting  $M=B$, by  \Corollary{\ref{cor:d0eqconcrete}}  we are
  in the above situation, and the functor $\Jo_+$  is given by: 
\begin{definition}\label{mls}  Define the functor
  \begin{displaymath}
 \Jo_H^G: \Mod_{\Dc_1}(B)\to \Mod_{\Dc_2}(B),\quad N \mapsto \Jo_H^G (N)=\Dc_2\cdot \Ann_{\Dc_1^-}(N). 
\end{displaymath}
We call $\Jo_H^G$ the differential  MLS-restriction.  \end{definition}

We record the  behaviour of the differential MLS-restriction in a chain of subgroups. 
\begin{lemma}\label{chain-lemma} Let $G_2 \subset G_1 \subset G$ be an inclusion of finite groups  and $V$ be a 
  representation of $G$. Then we have 
  \begin{displaymath}
    \Jo_{G_2}^G = \Jo^G_{G_1} \circ \Jo_{G_2}^{G_1}.
  \end{displaymath}
  
  \end{lemma}
  \begin{proof}
    Letting $\Dc=\Dc_B^{G}\subset \Dc_2=\Dc_B^{G_1}\subset \Dc_1=\Dc_B^{G_2}$ we have
    \begin{displaymath}
      \Jo^G_{G_1} \circ \Jo_{G_2}^{G_1}(N) = \Dc \Ann_{\Dc_1^-}(\Dc_1 \Ann_{\Dc_2^-}(N)) = \Dc
      \Ann_{\Dc_2^-}(N)=    \Jo_{G_2}^G(N),
    \end{displaymath}
  where the second equality follows from \Corollary{\ref{cor:d0eqconcrete}} since the $\Dc_1^0$-module $\Ann_{\Dc_2^-}(N)$
  is by restriction $\Dc_1^0 \subset \Dc_2^0$.
\end{proof}

The terminology is motivated by the fact that $\Jo_H^G$ is closely
related to Macdonald-Lusztig-Spaltenstein induction for group
representations, which is defined in the following manner
\cite{macdonald-weyl,Lusztig-Spaltenstein}; see also \cite[5.2]{geck},
and in particular for the construction of the representations of $S_n$
and the generalized symmetric group, see \cite[5.4]{geck} and
\cite{Ariki}. Suppose that $W$ is a representation of $H$ with
$W$-isotypic component $B_W$ in $B$. Let $\dop_W$ be the least integer
such that the homogeneous component ${B^{\dop_W}_W}$ of degree
$\dop_W$ is nonzero, and assume that ${B^{\dop_W}_W}$ is isomorphic to
$W$. Then the MLS-induced representation is defined as the isomorphism
class of the representation
\begin{displaymath}
\jo_H^G(W)= \Cb[G]B^{\dop_W}_W\subset B,
\end{displaymath}
and it is not difficult to prove that it is an irreducible $G$-representation (see [loc.\  cit.] or
\Proposition{\ref{prop:MLSrelation}} below).  In this way one gets a partially defined map
$\jo_H^G: \hat H\to \hat G$.

Note that it makes sense to extend the definition by dropping the condition
${B^{\dop_W}_W} \cong W$, but then $\Cb[G]{B^{\dop_W}_W}$ need not be irreducible. However, using
the $\Dc^0$-module structure one may still keep track of the decomposition of $\jo_H^G(W)$, as described
in the following proposition. We use the notation in \Definition{\ref{mls}}.
\begin{proposition}
  \label{prop:MLSrelation} Let $W$ be an irreducible representation over $H$ and $N$ be a
  $\Dc_1$-module such that $N\sim_H W$ \Prop{\ref{thm:montgomery}}. Assume that the restriction of
  $N^{ann}=\Ann_{\Dc_1^-}(N)$ to a $\Dc_2^0$-module is simple, and put
  $r=\dim_{\Cb}N^{ann}$. Then we have:
\begin{enumerate}
\item $\jo_H^G(W)=(W_1)^r$, where $W_1$ is an irreducible $G$-representation,
\item  $W_1\sim_G \Jo_H^G(N)$, so if $r=1$, then $\jo_H^G(W)\sim_G \Jo_H^G(N)$.
\end{enumerate}
\end{proposition}
Notice that the condition on $N^{ann}$  is trivially satisfied when $r=1$.
\begin{proof} That $N\sim_H W$ means that the $W$-isotypic component $B_W$ of $B$ is
    isomorphic to $W\otimes_\Cb N$ as a $\Dc_1[H]$-module, for some (simple) $\Dc_1$-submodule $N\subset B$
    \Prop{\ref{thm:montgomery}}.   Since $N^{ann}= N^{\dop_W}$ is a simple $\Dc_2^0$-module,
    \Theorem{\ref{prop:MLS}} implies that the $\Dc_2$-module $N_1=\Jo_H^G(N)=\Dc_2N^{\dop_W}\subset B$ is simple; hence by
    \Proposition{\ref{thm:montgomery}}
  \begin{displaymath}
    \Cb[G]N_1 \cong W_1\otimes_{\Cb} N_1
  \end{displaymath}
  for some irreducible $G$-representation $W_1$. In particular
  $\jo_H^G(W)=W_1\otimes_{\Cb} N^{\dop_W}$, since
  $N_1^{\dop_W}= N_1^{ann} = N^{ann}=N^{\dop_W} $.  This implies (1)
  and (2).  \end{proof}

One should note the slight conceptual difference: MLS-restriction of $\Dc$-modules as defined above
takes {\it submodules}\/ of $B$ to submodules of $B$, but MLS-induction of $G$-representations takes
(certain){\it isomorphism classes}\/ of irreducible representations to isomorphism classes of irreducible
representations. Our definition is partly motivated by the fact that we are interested in the actual
generators of the irreducible $\Dc$-submodules, not only the isomorphism classes. In the rest of the
section we will exemplify this for the generalized symmetric group.

\subsection{Simple  $R$-modules and partitions}
\label{section:partition}
Before we can describe MLS-restriction for the generalized symmetric group, we need to understand the
simple submodules of $B$ over the algebra
\begin{displaymath} 
  R_n=\Cb [\hf]^{S_n}
\end{displaymath}
that appeared in
\Proposition{\ref{gen-sym-R}}, mapping surjectively to $\Rc$.  

A multi-index is a function $\alpha : [n]= \{1, \dots , n\}\to \Nb$, and this determines the
monomial $x^\alpha= \prod x_i^{\alpha(i)}$.  Since $\nabla_i (x^\alpha) = \alpha(i)x^\alpha$ it
follows that the algebra $\Cb[\hf]$ acts multiplicity-free on $B$, where $M_\alpha=\Cb x^\alpha$ is
the unique simple submodule of $B$ in its isomorphism class. Therefore $M_\alpha$ is also an
$R_n$-module, necessarily simple, and $B$ is a semi-simple $R$-module. Using partitions, it is easy to
describe when $M_\alpha\cong M_\beta$ as $R_n$-modules.

The function $\alpha$ has fibres $P_\alpha (i) = \{j:\alpha (j)=i\}$, that induce a partition of the
set $[n]= \cup_{i\geq 0}P_\alpha(i)$. Note that some of the sets may be empty and that the order of
the sets in the partition is significant; we will call a sequence $P= (P_i)_{i=1}^r$\footnote{This
  should be regarded as an infinite sequence where $P_i= \emptyset$ when $i >r$, for some integer
  $r$.} such that $\cup P_i=[n]$ an {\it ordered partition}. Note also that $\alpha$ is determined
by $P_\alpha =(P_\alpha(i))_{i \geq 0}$.  Similarily, a sequence $(\lambda_1, \ldots, \lambda_r)$\footnote{Again this is regarded as an infinite sequence such that $\lambda_i =0$ when $i >r$.}, is
called an ordered partition of an integer $n$, denoted by
$\lambda = (\lambda_i)_{i\geq 1} = (\lambda_1, \ldots, \lambda_r) \vdash^o n$, if the integers
$\lambda_i \geq 0$ and $\lambda_1 + \cdots + \lambda_r =n$.  Given an ordered partition
$P= (P_i)_{i=1}^s$ of the set $[n]$, the partition
$\lambda^P= (\lambda^P_1, \dots , \lambda^P_s) \vdash^o n$ is defined by $\lambda_i^P = |P_i|$.  In
particular, we put $\lambda^\alpha = \lambda^{P_\alpha}$, and say that $\lambda^\alpha$ is the
ordered partition of the integer $n$, associated to $P_\alpha$.
  
The unordered partition $\bar P = \{P_{i_1}, \ldots , P_{i_r}\}$ of $[n]$ corresponding to an
ordered partition $P= (P_i)_{i \geq 1}$ is the set of subsets such that $P_i \neq \emptyset$.
Similarly, the unordered partition
$\bar \lambda = \{\lambda_{i_1}, \ldots , \lambda _{i_r} \} \vdash n$ of an ordered partition
$\lambda \vdash ^0 n$ is the set of nonzero elements in the sequence $\lambda$.\footnote{If one
  thinks of the ordered partition of an integer $n$ as a sequence of columns with
  $\vert P_\alpha(0)\vert, \vert P_\alpha(1)\vert$ boxes, the relation between the concepts of
  ordered partitions $P$ and integers $\lambda \vdash^o n$ is similar to the one between Young
  tableaux and Young diagrams \cite[2.1]{sagan}.}

The action of the symmetric group on the set $[n]$ induces an action on the set of multiindices
$\alpha$ by $\sigma\cdot \alpha = \alpha \circ \sigma^{-1}$, so that
$P_{\sigma \cdot \alpha}(i) = \sigma (P_{\alpha}(i))$. Clearly, then $P_\alpha$ and $P_\beta$ belong
to the same orbit under the symmetric group if and only if
$\vert P_\alpha(i)\vert=\vert P_\beta(i)\vert$ for $l=0,1,\ldots,$ , that is, exactly when the
induced partitions of numbers $\lambda^\alpha=\lambda^\beta$.

 \begin{prop}\label{r-iso} 
   \begin{enumerate}
   \item For any simple $R_n$-module $M\subset B$ there is a multiindex $\alpha$ such that
     $M\cong M_\alpha$.  There is an isomorphism $M_\alpha\cong M_\beta$ if and only if there exists
     $\sigma\in S_n$ such that $\sigma \cdot \alpha = \beta\iff \lambda^\alpha=\lambda^\beta$.
   \item  The decomposition of the $R_n$-module $B$ into isotypical components is
   \begin{displaymath}
     B= \bigoplus_{\lambda \vdash^o n}B_{\lambda}
   \end{displaymath}
   where for each ordered partition $\lambda$ we have an isotypical
   component $B_{\lambda}$ of the form
\begin{displaymath}
  B_{\lambda} = \bigoplus_{\alpha  \in \Omega_\lambda} \Cb x^\alpha,
\end{displaymath}
and $\Omega_\lambda=\{ \alpha: [n]\to \Nb \ \vert  \ \lambda^\alpha=\lambda\}$.

 \end{enumerate}
 \end{prop}
  \begin{proof}
   (1): 
   The first assertion is already motivated. The mapping $x^\alpha$ to $x^{\sigma\cdot \alpha}$
   defines an isomorphism $M_\alpha \to M_{\sigma \alpha}$ of $R_n$-modules. Conversely, if
   $M_\alpha \cong M_\beta$, then, using the description of $\Rc$ in \Proposition{\ref{gen-sym-R}},   $p(\alpha(1), \dots , \alpha(n)) = p(\beta(1), \dots , \beta(n))$
   for all symmetric polynomials $p(x_1, \dots , x_n)\in \Cb[x_1, \dots , x_n]^{S_n}$.  This implies
   that $\beta (i) = \alpha (\sigma(i))$ for some $\sigma \in S_n$.

   (2): Immediate from (1).
 \end{proof}

 The set $\Omega_\lambda$ is the set of $S_n$-orbits of multi-indices $\alpha : [n]\to \Nb$ for a
 given $\lambda = \lambda^\alpha$. Select an ordered partition $P^\lambda=(P^\lambda_i)_{i=1}^r$ of
 $[n]$ such that $|P^\lambda_i|= \lambda_i$. Then $\Omega_\lambda \cong S_n/G_P$, where
 $G_P= \prod S(P^{\lambda}_i)$ and $S(P^{\lambda}_i)$ is the symmetric group of the set
 $P^{\lambda}_i$.

 \subsection{Decomposition of B for some complex reflection groups}\label{gen-symm-modules}
 There is a striking use of $\Rc$ for the construction of simple
 $\Dc_B^G$-modules for generalized symmetric groups
 $G=G(d,1,n)= A(d,1, n)\rtimes S_n = A\rtimes S_n$.  The main point of
 our proof is that the description of the simple modules of $\Rc$ in
 \Proposition{\ref{r-iso} } makes it easy to determine when, for
 simple $\Dc_B^{H}$ modules $L_1\not \cong L_2$, we have
 $\Jo_{G_P}^G(L_1)\not \cong \Jo_{G_P}^G(L_2) $, where
 $H = A \rtimes G_P$ and $G_P$ is a Young subgroup of $S_n$. For the
 equivalent group representation case (using
 \Proposition{\ref{thm:montgomery}}), the results about the symmetric
 group ($d=1$) go back to Specht \cite{specht:dieirrede}, see also
 \cite{peel}, and for the generalized symmetric group ($d>1$) see
 \cite{Ariki}.

To a multi-index $\alpha : [n]\to \Nb$ we have associated an ordered partition $P_\alpha$ of $[n]$
(\ref{section:partition}).  Say that an {\it unordered}\/ partition $P= \{P_{ij}\}$,
$[n]= \cup_{i, j} P_{ij}$ is an $\alpha$-partition if $\cup_j P_{ij} = P_\alpha(i)$. Let
$\lambda^P\vdash n$ be the (unordered) integer partition that is determined by $P$, so that $\lambda^P$ can be
visualized by a sequence of at most $n$ Young diagrams, each one of cardinality
$\vert P_\alpha(i)\vert$.

Let $S(\Omega)$ be the symmetric group of a subset $\Omega $ in $[n ]$.  Given a multi-index
$\alpha: [n]\to \Nb$ we put $G^\alpha =A \rtimes \prod_i S(P_\alpha (i))$ and given an
$\alpha$-partition $P$ we put $G^\alpha_P = A \rtimes \prod_{ij} S(P_{ij})\subset G^\alpha $.
\begin{proposition} 
\label{prop:specht}
\begin{enumerate}
\item The simple $\Dc_B^A$-submodules of $B$ are of the form
  $N_\alpha = B^A x^\alpha$, where $\alpha : [n]\to [d-1]$. If
  $N_\beta = B^A x^\alpha$, $\beta : [n]\to [d-1]$, is another such
  module, then $N_\beta \cong N_\alpha \iff \alpha = \beta $.
\item Let $P $ be an $\alpha$-partition and define the polynomial $s^\alpha_{P}= s_{P}x^\alpha$, where $s_{P}$ is the
  Jacobian of the map $B^{G^\alpha_P}\to B^A$. Then
  \begin{displaymath}
    N_P^\alpha= \Dc_B^{G^\alpha_P}s^\alpha_{P} = B^{G^\alpha_P} s^\alpha_{P} 
\end{displaymath}
is a simple
  $\Dc_B^{G^\alpha_P}$-module.
\item The module $M^\alpha_P=\Jo_{G^\alpha_P}^G(N_P^\alpha)=\Dc_B^G s^\alpha_{P}$ is a simple
  $\Dc_B^G$-submodule of $B$.
\item Let $\beta : [n]\to [d-1]$ be another multi-index and $Q$ be a $\beta$-partition. Then
  \begin{displaymath}
    M^\alpha_{P}\cong M^\beta _{Q}\iff \lambda^P = \lambda^Q \text{ and } \beta \in S_n \cdot \alpha.
\end{displaymath}
\item Let $M$ be a simple $\Dc_B^G$-submodule of $B$. Then there
  exists a multiindex $\alpha : [n]\to [d-1]$ and an $\alpha$-partition
  $P$ such that $M\cong M^\alpha_P$.
   \end{enumerate}
\end{proposition}
\begin{remark} \label{specht-explicit}
  \begin{enumerate}
  \item The group $\prod S(P_\alpha(i))$ is the inertial group of the $\Dc_B^A$-module $N_\alpha$
    with respect to the homomorphism $B^{G^\alpha}\to B^A$ (see \cite{kallstrom-directimage})
    and subgroups of the form $\prod S(P_{ij}) $ are its parabolic subgroups, i.e. \ subgroups that
    preserves some closed point in $\Spec B^A$.  \item Let
      $G_P= \prod S(P_j) \subset S_n$ be the Young group of a partition $P= \{P_j\}$ of $[n]$. The
      Jacobian of the invariant map $B^{G_P}\to B$ is independent (up to a multiplicative constant)
      of the choice of homogeneous coordinates in the polynomial ring $B^{G_P}$; in the calculation
      below we will use that it may be taken as
      $s_P=\prod_{j}\prod_{k < l\in P_{j}}(x_k-x_l)$.  Similarly, for a 2-step partition
      $P=\{P_{ij}\}$, we may take the van der Monde determinants
  $$s_P=\prod_{ij}\prod_{k <  l\in P_{ij}}(x_k^d-x_l^d).$$
\end{enumerate}
\end{remark}

A key step in the proof of \Proposition{\ref{prop:specht}} is the following relation between an
integer partition $\lambda^P\vdash n$ that comes from a partition $P$ of the set $[n]$ and the
partition $\lambda^\alpha\vdash n$, where $\alpha$ is a multiindex that occurs in an expansion of
the Specht polynomial $s_P$ of $P$.  In our context, the result identifies which isomorphism class
of $\Rc$-modules the Specht polynomial corresponds to. 

Let $\lambda^c\vdash n$ denote the conjugate of a partition $\lambda \vdash n$, i.e.\ the partition
whose $i$th part $\lambda^c_i$ is the number of $m$ with $\lambda_m \geq i$.  The conjugate of a
partition $P=\{P_i\}$ is $P ^c = \{Q_i\}$, where
$Q_i= \{j \in [n] \ \vert \ j\in P_s \text{ belongs to at least $i$ different } P_s \}$, so that
$\lambda_P^c = \lambda_{P^c}$.
\begin{lemma}\label{keypartition} Let $x^\alpha$ be a non-zero term  in an expansion of $s_P$. Then 
  $(\lambda^{\alpha})^c = \lambda^P$.  More precisely, in the notation of \Proposition{\ref{r-iso}},
  we have $s_P \in B_\lambda$, where the ordered partition
  $\lambda = (|Q_1|,\ldots , |Q_s| )\vdash^0 n$ and $Q$ is the conjugate of $P$.
\end{lemma}
\begin{proof}  If $P= \{P_j\}_{j=1}^r$ we write $\lambda^P = \{n_1, \ldots , n_r\}\vdash n$, where
  $n_j = |P_j|$. Expanding the Specht polynomial of one subset $P_j$
\begin{displaymath}
  \prod_{1\leq k<l\leq n_j} (x_k-x_l) = \sum c_{\alpha_i} x^{\alpha_i}, \quad \alpha_i: [n_j]\to \Nb,
\end{displaymath}
then if $c_{\alpha_i} \neq 0$,  it follows that the (unordered) set
$\{\alpha_i (1),\alpha_i (2), \ldots , \alpha_i(n_j) \} = \{n_j-1, n_j-2, \ldots , 1, 0\}$;  hence
$\lambda^{\alpha_i}=\{1,1,\ldots , 1\}\vdash n_j$ (so that $\lambda^{\alpha_i} = \{n_j\}^c$).    Since
\begin{displaymath}
  s_P= \prod_{j=1}^r s_{P_j} = \prod_{j=1}^r(\sum c_{\alpha_i} x^{\alpha_i})= \sum a_\alpha x^\alpha
\end{displaymath}
it follows that if $a_\alpha \neq 0$, then $\lambda^\alpha_i = |\{l  \vert \ \alpha (l)= i \}|$ is the number of subsets $P_j$ with
$|P_j| > i$.  This implies that $\lambda^\alpha =( \lambda^P)^c$.
\end{proof}
\begin{pfof}{\Proposition{\ref{prop:specht}}}
  (1): We have
  \begin{displaymath}
    B^A = \Cb [x_1^d, \ldots, x_n^d]\subset \Dc_B^A= \Cb [x_1^d, \ldots , x_n^d, \partial_1^d, \ldots
    , \partial_n^d ],
\end{displaymath}
so that
$\Cb[\partial_1^d, \ldots , \partial_n^d ]\subset (\Dc_B^A)^- \subset
\sum_{i=1}^n \Dc_B^A \partial_i^d$
and hence
$\Ann_{(\Dc_B^A)^-} (B) = \{x^\alpha \ \vert \ \alpha :[n]\to
[d-1]\}$.
Notice that $\Cb[\nabla_1, \ldots , \nabla_n]\subset \Dc_B^A$.  If
$\Cb x^\alpha \cong \Cb x^\beta$ as
$\Cb[\nabla_1, \ldots , \nabla_n]$-modules, then $\alpha =
\beta$.
Therefore $B= \oplus \Dc_B^A x^\alpha$, where the sum runs over
multi-indices $\alpha : [n]\to [d-1]$, and $N_\alpha \cong N_\beta$
implies $\alpha = \beta$.  It is straightforward to see that $B^A$ is
a simple $\Dc_B^A$-module, implying that each $N_\alpha$ is also
simple.

Another way to see that $N_\alpha$ is a simple $\Dc_B^A$-module is to appeal to $A$-semiinvariants,
which is in a sense more easy to see. This is what we will have to do in (2) below.

  (2): The element $s_{P}x^\alpha $ defines a $G_P^\alpha$-semiinvariant
  $\chi : G_P^\alpha\to \Cb^* $, and generates the $B^{G_P^\alpha}$-module of all semi-invariants associated to $\chi$. It then follows from
  \Proposition{\ref{thm:montgomery}} that 
  $N_P^\alpha= \Dc_B^{G^\alpha_P} s_{P}x^\alpha = B^{G_P^\alpha}s_{P}x^\alpha$ and that this is a
  simple $\Dc_B^{G^\alpha_P}$-module.

  (3): Put $\Dc_1= \Dc_B^{G^\alpha_P} $ and $\Dc_2 = \Dc_B^G$. Then $\Cb s^\alpha_{P}$ forms a
  1-dimensional $\Dc_1^0$-module, so the assertion follows from \Theorem{\ref{prop:MLS}}, (1).

(4): If $ M^\beta _{Q}  \cong M^\alpha_{P} $, then
\begin{displaymath}
  \Cb s_{Q} x^\beta  \cong \Cb s_{P} x^\alpha 
\end{displaymath}
as $ R_n$-modules. Since $s_P,s_Q \in B^A $ it follows from  (1) and \Proposition{\ref{r-iso}} (1) that there
exists $\sigma \in S_n$ such that $ \beta \equiv \sigma \alpha (\omod d) $, and since
$\alpha, \beta : [n]\to [d-1] $ this implies that $\beta = \sigma \alpha $.  Therefore we have
isomorphisms, where the second one comes from the action of $\sigma$,
\begin{displaymath}
 \Cb s_{Q} x^\beta \cong   \Cb s_{P} x^\alpha  \cong   \Cb s_{\sigma \cdot P} x^\beta,  
\end{displaymath}
and hence $ \Cb s_{Q} \cong \Cb s_{\sigma \cdot P} \cong \Cb s_P $; hence by 
\Lemma{\ref{keypartition}},  $\lambda^Q= \lambda^P$.

Conversely, if $\lambda^Q= \lambda^P$ and $\beta \in S_n \cdot \alpha$, then there exists
$\sigma \in S_n$ such that $\beta = \sigma \alpha$ and $Q= \sigma P$, and hence the
$\Dc_n$-homomorphism $\sigma : B \to B$ induces an isomorphism $ M^\alpha_{P} \cong  M^\beta _{Q}$.

(5): First note that if $\lambda^{\alpha}\not = \lambda^{\beta}\vdash n$, then
$\beta \not \in S_n \cdot \alpha $, so that the corresponding modules are non-isomorphic.  It
follows that the set of nonisomorphic simple $\Dc_B^G$-submodules that arise in (4) is parametrised by
a sequence of at most $d$ Young diagrams.  This agrees with the parametrisation of the set of
conjugacy classes $\Clo (G(d,1,n))$ \cite{osima-gen-symm}.
\end{pfof}

Consider now $G=G(m,e,n)= A\rtimes S_n$ with $e>1$, and put $\Dc= \Dc_B^{G}= (\Dc^A_B)^{S_n}$. A
possible strategy to construct the simple $\Dc$-submodules of $B$ is as follows. Let
$\{N_i\}_{i=1}^r$ be a set of representatives of the simple $\Dc_B^A$-submodules of $B$ (this is
done for $A= A(e,e,2)$ in \Section{\ref{cyclic-group}}). Let $G_i\subset S_n$ be the inertial group
of $N_i$, $i=1, \ldots , r$, and put $\bar G_i = A\rtimes G_i$. Let $\{M_{ij}\}$ be a set of
representatives of the simple $\Dc_B^{\bar G_i}$-submodules of $B$. To find such modules $M_{ij}$ it
is natural to consider parabolic subgroups $\bar G^P_i \subset \bar G_i $, let $s_{P,i}$ be the
Jacobian of the invariant map $B^{\bar G^P_i}\to B$, and expect that $M_{ij}$ is contained in the
composition series of $\Jo^{\bar G_i}_{\bar G^P_i}(\Dc_B^{\bar G^P_i} s_{P,i})$.

Then we can construct the $\Dc$-module
  \begin{displaymath}
\bar M_{ij}  = \Jo^G_{A}(N_i)\otimes_{B^G} \Jo^G_{G_i}(M_{ij}).
  \end{displaymath}
  Since $B$ is semisimple over $\Dc$ and $B$ is free over $B^G$, $G$ being a reflection group, it
  follows that the $\Dc$-submodules $\Jo^G_{G_i}(M_{ij}) $ and $ \Jo^G_{A}(N_i)$ also are free over
  $B^G$. An interesting problem would be to understand the decomposition of $\bar M_{ij}$ into
  simples. When $e=1$, so that $B^A$ is polynomial ring, we are in the situation of
  \Proposition{\ref{prop:specht}} where these modules are already studied, albeit expressed
  differently.

\section{The branch rule for $S_n$}
\label{branchrule}
We start with a fairly general condition in (\ref{gen-symm}) that ensures that the restriction of a
simple $\Dc_1$-module to a module over a subring $\Dc_2 \subset \Dc_1$ is multiplicity free. In
(\ref{symm-group}) we give a proof of the classical branching rule for the symmetric group $S_n$,
expressed in terms of $\Dc$-modules and based on a lowest weight argument, where $\Dc= \Dc_B^{S_n}$.
In (\ref{branching-graph}) and ({\ref{canonical-section}}) we discuss the branching graph of the
$\Dc$-module $B$ and provide $B^{ann}$ with canonical bases, which turn out to coincide with Young
bases.

\subsection{The generalized symmetric groups}\label{gen-symm}
The generalized symmetric group $G_n= G(m,1,n)= A(m,1,n)\rtimes S_n$ acts on
$B= \Cb[x_1, \dots , x_n]$ by permuting the coordinates and by multiplying by $m$th roots of
unities. It contains the subgroup $G_{n-1}\subset G_n$ of elements that fix the variable $x_n$.  Let
$\Dc_n= \Dc_B^{G_n}$ be the ring of invariant differential operators, so that
$\Dc_n\subset \Dc_{n-1}$. Letting $B_{n-1}= \Cb[x_1, \dots , x_{n-1}]$, we note that 
\begin{equation}
\label{eq:n-1}
 \Dc_{n-1}= \widetilde  \Dc_{n-1}[x_n,\partial_n],
\end{equation}
where  $\widetilde  \Dc_{n-1}=\Dc_{B_{n-1}}^{G_{n-1}}$.

The branch rule for representations of the generalized symmetric groups, describing induction from
$G_{n-1}$ - to $G_n$ -representations, is the second statement in \Proposition{\ref{gen-branch}}
below.  By \Proposition{\ref{thm:montgomery}} it is equivalent to the first statement on restriction
from $\Dc_{n-1}$- to $\Dc_n$-modules, which we will see is a consequence of the determination of
$\Rc$ for the generalized symmetric group in \Proposition{\ref{gen-sym-R}}.\footnote{The
  correspondence between operations like induction and restrictions for representations and direct
  and inverse images of $\Dc$-modules is described in more detail in
  \cite{kallstrom-directimage}.}

\begin{prop}\label{gen-branch}
\label{quick} Let $N$ be a simple $\Dc_{n-1}$-submodule of $B$ and  $V$ be
    the corresponding irreducible $G_n$-representation, so that $V\sim_{G_{n}} N$ in
    \Proposition{\ref{thm:montgomery}}.
  \begin{enumerate}
 \item The restriction of $N$ to a $\Dc_{n}$-module $\reso(N)$ by  the inclusion 
 $\Dc_n\subset \Dc_{n-1}$  splits into a direct sum
 \begin{displaymath}
   \reso(N) =\bigoplus M_i
 \end{displaymath}
 of pairwise non-isomorphic simple  submodules.
 \item   The induced representation $\ind_{G_{n-1}}^{G_n} (V)$ splits into a sum 
   \begin{displaymath}
     \ind_{G_{n-1}}^{G_n} (V) = \bigoplus W_i,
   \end{displaymath}
of pairwise non-isomorphic irreducible representations, where $W_i\sim_{G_n} M_i$.
 \end{enumerate}
\end{prop}
We put $\Dc_B^{x_n}(k) = \{P \in \Dc_B \ \vert \ [\nabla_n, P] = k P \}$ and
$\Dc_B^{x_n^-} = \oplus_{k<0} \Dc_B^{x_n}(k) $.  Let
$\bar \af_n = \sum_{k \geq 0, l\geq 0, m \vert l-k} \Cb p^{(n)}_{k,l}$ be the Lie algebra in
\Proposition{\ref{symm-lie}}, where we have the power differential operator
$p^{(n)}_{k,l}= \sum_{i=1}^n x_i^{k} \partial_i^l$.  We notice here the following:
\begin{enumerate}[label=(\roman*)]
\item $x_n, \partial_n\in \Dc_{n-1}$,
\item $\bar \af^-_{n-1}\subset \tilde \Dc_{n-1}^- \subset \Dc_{n-1}^{-}\subset \Dc_{n-1} \bar \af^-_{n-1}+ \Dc_B\partial_n$,
\item $\bar \af^-_{n-1} \subset \Dc_n^-+\Dc_B^{x_n^-}$,
\end{enumerate}
where (ii) is a consequence of \Proposition{\ref{prop:pbw}}, and (iii) follows from the relation
$p^{(n)}_{k,l} = p^{(n-1)}_{k,l}+ x_n^k\partial_n^l$.

\Proposition{\ref{gen-branch}} is a consequence of (i-iii) and the
fact that the algebra $\Rc = \Dc_n^0/(\Dc_n^0\cap (\Dc_n \Dc_n^-))$ is
commutative. In the theorem below, which encodes this argument, we
consider general graded subrings
$\Dc= \oplus_{k\in \Zb} \Dc(k)= \Dc^+\oplus \Dc^0\oplus \Dc^-\subset
\Dc_B$ as in \Section{\ref{abstract-eq}}.

\begin{theorem}\label{branch-diff} Let $\Dc_2 \subset \Dc_1$ be an inclusion of graded subrings of
  $\Dc_B$ as above, where $x_n,\partial_n\in \Dc_{1}$, and put
  $\tilde \Dc_1= (\Dc_1)^{\partial_n}= \{P \in \Dc_1 \ \vert \ [\partial_n, P]=0\}$. Assume that
  there exists a graded Lie subalgebra $\af_1$ of $ \tilde \Dc_1$ such that:
  \begin{enumerate}
  \item $\af^-_1\subset  \tilde \Dc_1^- \subset  \Dc^-_1\subset \Dc_1 \af_1^- + \Dc_B\partial_n$,
  \item $\af^-_1 \subset  \Dc_2^- +\Dc_B^{x_n^-}$.
  \end{enumerate}

Let $N$ be a simple $\Dc_{1}$-submodule of $B$ such that $\dim_{\Cb} \Ann_{\Dc_1^-}(N)=1$ and assume
that for any simple $\Dc_2$-submodule $M\subset N$ we have $\dim_{\Cb} \Ann _{\Dc_2^-}(M)=1$; these
two conditions are satisfied if $\Rc_1 $ and $\Rc_2$ are commutative
($\Rc_i = \Dc_i^0/\Dc_i^0 \cap (\Dc_i \Dc_i^-)$). Then it follows that the restriction $\reso(N)$ to
a $\Dc_{2}$-module splits into a direct sum
 \begin{displaymath}
   \reso(N) =\bigoplus M_i
 \end{displaymath}
 of pairwise non-isomorphic simple submodules. 
  \end{theorem}

  \begin{remark} 
    \begin{enumerate}
    \item It would be interesting to find applications of \Theorem{\ref{branch-diff}} in other
      situations. Let $H \subset G\subset \Glo(V)$ be an inclusion of finite groups.  Put
      $\Dc_2= \Dc_B^G \subset \Dc_1=\Dc_B^H$, $\Rc_H= \Dc_1^0/\Dc_1^0 \cap (\Dc_1 \Dc_1^-)$,
      $\Rc_G= \Dc_2^0/\Dc_2^0 \cap (\Dc_2 \Dc_2^-)$.  Assume that $H$ fixes the variable $x_n$ and
      that $\af_1$ is a graded Lie subalgebra of $\Dc_1$ such that (1) and (2) in
      \Theorem{\ref{branch-diff}} are satisfied. If now $\Rc_{H}$ and $\Rc_G$ are commutative, it
      follows as in the proof below that any irreducible representation of $G$ restricts to a
      multiplicity free representation of $H$.
    \item Note that
      $\Rc_1= \widetilde \Dc_1^0/\ \widetilde \Dc_1^0 \cap ( \widetilde\Dc_1 \widetilde\Dc_1^-)$,
      since $\partial_n\in \Dc^-_1$.
    \end{enumerate}

  \end{remark}

\begin{pfof}{\Proposition{\ref{gen-branch}}}  
  We know that the algebras $\Rc_i$ are commutative by
  \Propositions{\ref{symm-lie}}{\ref{prop:pbw}}, and the above remark. Putting
  $\af_1 = \bar \af_{n-1}$, (i) implies the first and (ii-iii), where the ring $\tilde \Dc_{n-1}^-$
  is a subring of $\tilde \Dc_1$, implies the conditions (1-2) in \Theorem{\ref{branch-diff}}, hence
  we get (1).  The corresponding assertion (2) for representations follows since
  $\ind_{G_{n-1}}^{G_n}(V) \sim_{G_n} \reso(N)$ \Prop{\ref{ind-res}}.
   \end{pfof}

  \begin{pfof}{\Theorem{\ref{branch-diff}}}
    If $\Rc_i$ is commutative and $M$ is a simple $\Dc_i$-module, then \Theorem{\ref{main}} implies
    that $\dim_{\Cb} \Ann_{\Dc_i^-}(M) =1$.  So assume $\Ann_{\Dc_{1}^-}(N)= \Cb y$ for some
    homogeneous elements $y\in N$. Similarly for any simple $\Dc_2$-submodule $M\subset \reso (N)$
    we have $\Ann_{\Dc_2^-}(M)=\Cb z$, for a homogeneous polynomial $z\in M$, and we then put
    $\deg(M) = \deg (z)$.  Since $\nabla \in \Dc_2$ it follows that if $M_1$ is another simple
    submodule of $N$ and $\deg (M)\neq \deg (M_1)$, then $M \not \cong M_1$.  Conversely, we will
    prove that if $\deg (M)= \deg (M_1) $, then $M= M_1$.  Expand $z=y_0+y_1x_n+\cdots +y_ax_n^a$, where
    $\partial_n(y_i)=0$, $y_i\in N_1$ since $x_n, \partial_n \in \Dc_1$, and $y_a\neq 0$. Let
    $r_1\in \af^-_1 $, so that by (2), $r_1= r_2+ r^{(n)}$, where $r_2\in \Dc_2^-$,
    $r^{(n)}\in \Dc_B^{x_n^-}$. Since $r_2(z) =0$, we have (as detailed below)
  \begin{displaymath}
 r_1(z) =  r_1(y_a)x_n^a + (\text{l.o.\ in } x_n)= r^{(n)}(y_0+y_1x_n+\cdots +y_ax_n^a).
\end{displaymath}
The first equality follows since $r_1 \in \tilde \Dc^-_1$, so that
$r_1 = \sum c_{\gamma, b}(x')^\gamma(\partial')^\beta \partial_n^b$,
where $x' = (x_1, \ldots , x_{n-1})$,
$\partial' = (\partial_1, \ldots , \partial_{n-1})$,
$|\beta| + b >| \gamma |$, and $\partial_n (y_a)=0$. Therefore
$r_1(y_a)=0$. Hence by (1), $y_a\in \Ann_{\Dc_1^-}(N)$, and therefore
$\Cb y_a = \Ann_{\Dc_{1}^-}(N)=\Cb y$.  We have also
$\Ann_{\Dc_2^-}(M_1)=\Cb z'$ for some homogeneous polynomial $z'$, and
it suffices now, since $M$ and $M_1$ are simple, to prove that
$\Cb z = \Cb z'$ when $\deg z'=\deg z$.  We expand
$z'=y'_{a'}x_n^{a'}+(\text{l.o. in } x_n)$ in the same way as $z$
above, and by the same argument as before we have
$\Cb y'_{a'}=\Ann_{\Dc_{1}^-}(N)=\Cb y_a$.  Therefore
$\deg (y'_{a'}) = \deg (y_a)$, and as $\deg (z)= \deg (z')$, it
follows also that $a= a'$.  Multiplying $z'$ by a complex number so
that $y'_{a'}= y_a$, it suffices now to see that $z= z'$.  Assume on
the contrary that
\begin{displaymath}
z-z' = y'_bx_n^b+(\text {l.o. in  } x_n)\neq 0.
\end{displaymath}
Since $z$ and $z'$ are homogeneous of equal degree it follows that $\deg (z-z')= \deg (z)$. Clearly
$b<a$, and again we have $y'_b \in \Ann_{\Dc_{1}^-}(N)$ so that $\deg (y'_b) = \deg (y_a)$,
implying that $\deg (z-z') < \deg (z)$, which is a contradiction. Therefore $z=z'$.
  \end{pfof}

\subsection{The symmetric group}\label{symm-group}
The symmetric group $S_n$ is a subgroup of $G_n$, so we have actions
of $S_n$ and its subgroup $S_{n-1}$ on both $B$ and $\Dc_B$, and we
now put instead
$\Dc_n = \Dc_B^{S_n}\subset \Dc_{n-1}= \Dc_B^{S_{n-1}}$. We want to
describe the decomposition (1) in \Proposition{\ref{quick}} in more
detail when $m=1$, which, by \Proposition{\ref{thm:montgomery}}, also
implies the very well-known branching rule for the symmetric group.
\begin{remark}
  In \citelist{\cite{peel}\cite{ james-kerber}*{Th. 2.4.3}} the proof of the branching rule for
  representations of the symmetric group requires the non-trivial fact
  that the standard Specht polynomials of shape $\lambda\vdash n$
  form a basis of a simple $S_n$-module $V_\lambda$. The proof below
  is instead based on \Corollary{\ref{cor:d0eqconcrete}} and
  \Proposition{\ref{prop:specht}}, where the latter shows that any
  simple $R_n$-submodule of $B^{ann}$ is isomorphic to $k s_P$ for
  some Specht polynomial $s_P$. The fact that the standard Specht
  polynomials, indexed by the standard Young tableaux, form a basis is
  then an immediate consequence of the branching rule, as described in
\Section{\ref{branching-graph}}.
\end{remark}

We assume now that every partition $ \lambda \vdash n$ is {\it ordered},
defining a  function $\lambda: \{1, 2, \ldots \}\to \Nb $ such that
$\lambda (i)\geq \lambda (i+1)$; this is the same as associating a
Young diagram to $\lambda$.


We already know that $B^{ann}$ is a Gelfand module, but we can be more
precise. 
\begin{corollary} The $S_n$-representation $B^{ann}$  is multplicity free
and is canonically  decomposed
  \begin{displaymath} 
    B^{ann}= \bigoplus_{\lambda\vdash n } V_\lambda,
  \end{displaymath}
  where   the representations $V_\lambda$   are irreducible and of the form
  \begin{displaymath}
  V_\lambda =k[S_n]s_P= \{p\in B\ \vert \ \sum_{i=1}^n x_i^k \partial_i^l p =0, 0 \leq k
  < l \leq n-1, \text{ and } p =\sum_{\lambda^\alpha = \lambda^c}
  c_\alpha x^\alpha 
   \},
\end{displaymath}
where  $s_P$ is the  Specht polynomial of a partition $P$ of $[n]$ such that
$\lambda^P = \lambda$ (see \Remark{\ref{specht-explicit}}).
\end{corollary}
\begin{proof}
  By \Proposition{\ref{prop:specht}} any simple $\Dc_n$-submodule of
  $B$ is isomorphic to a module of the form $\Dc_n s_P$ and
  $\Dc_n s_P \cong \Dc_n s_Q$ for another partition $Q$ of $[n]$ if
  and only if $\lambda = \lambda^Q$. If $x^\alpha $ is a monomial term
  in $s_P$ then $\lambda^{\alpha} = \lambda^c$
  \Lem{\ref{keypartition}}. By \Proposition{\ref{thm:montgomery}} it
  follows that $B^{ann}= \oplus_{\lambda \vdash n} k[S_n] s_P$. Putting
  $V_\lambda = k[S_n] s_P $, where $\lambda = \lambda_P$, it follows
  that any momonial term $x^\alpha$ in any polynomial that belongs
  to $V_\lambda$ satisfies $\lambda^\alpha= \lambda^c$. The description
  of $B^{ann}$ follows from \Proposition{\ref{symm-lie}}, (2).
\end{proof}

One says that a box in a Young diagram is {\it addable} if one gets
another Young diagram by adding a box.
  \begin{theorem}\label{branch2} Let $N_{\lambda}= \Dc_{n-1}
    v_\lambda$
    be a simple $\Dc_{n-1}$-module corresponding to the partition
    $\lambda \vdash  n-1$ in \Proposition{\ref{prop:specht}}, where
    $k v_\lambda = \Ann_{\Dc_{n-1}^-}(N_\lambda)$, and let 
    $ \reso (N_{\lambda})$ denote its restriction to $\Dc_n$-module. 
    \begin{enumerate}
    \item Then
      \begin{displaymath}
        \reso (N_\lambda) =\bigoplus_\mu N_{\mu}
      \end{displaymath}
      is multiplicity-free and the simple direct composants $N_{\mu}$
      correspond to all partitions $\mu \vdash n$ that can be
      formed by adding a box to $\lambda$.
    \item Assume that $N_\mu$ corresponds to adding a box to the $r$th
      row of the Young diagram of $\lambda$. Then the decomposition in
      (1) is determined by submodules $N_\mu\subset N_\lambda$ where
      $N_\mu$ is generated by a homogeneous lowest weight polynomial
      of the form
      $v_\mu = x_n^a v_\lambda + v_a' \in \Ann_{\Dc_{n}^-}(N_\lambda)
      $,
      where $a= \lambda(r)$. We have
      $\deg (v_\mu)= \sum_i i \lambda (i) + a - l$, where $l$ is the
      number of $i$ such that $\lambda (i)\neq 0$.
    \end{enumerate}
 \end{theorem}

 That a box is addable means more precisely the following. Given an
 ordered partition $\lambda \vdash n-1$ and an integer $r$ one gets
 the function $\mu: \{1, 2, \ldots \}\to \Nb $ by
 $\mu(r) = \lambda(r)+1 $, $\mu(i)= \lambda(i)$, $i\neq r$, then the
 index $r$ is addable if $\mu$ again is non-increasing. We note that
 an index $r$ is addable to $\lambda\vdash n-1$ if and only if the
 index $a= \lambda (r)$ is addable to the conjugate partition
 $\lambda^c\vdash n-1$; we then also write $a\in \lambda$.

 Define the bilinear form
 $\langle \cdot, \cdot \rangle : B\otimes_{\Cb} B\to \Cb, \langle f ,g \rangle = (f(\partial )
 g)\vert_{x=0}$,
 so that the monomials form an orthogonal basis with respect to $\langle \cdot, \cdot \rangle $, and
 define the antiisomorphism $t: \Dc_B \to \Dc_B$ by $x^t_i= \partial_i$, $\partial_i^t = - x_i$, and
 $(PQ)^t = Q^t P^t$. Then we have:
 \begin{enumerate}[label=(\roman*)]
 \item   $\langle P f ,g \rangle = \langle f , P ^tg \rangle $, $f,g \in B$, $P \in \Dc_B$.
 \item $(\Dc^0)^t = \Dc_n^0$, $(\Dc_n^+)^t = \Dc_n^-$, and $\Dc_n^t = \Dc_n$.
 \item If $f$ and $g$ are homogeneous of different degrees, then $\langle f ,g \rangle =0$. More
   precisely, if $f_i= \sum_{\alpha \in \Gamma_i} c_\alpha x^\alpha \in B $, $i=1,2$, where the set
   of multiindices $\Gamma_ 1\cap \Gamma_2 = \emptyset$, then $\langle f_1 ,f_2 \rangle =0$.

\item $\langle  f , f \rangle \neq 0$, $f\in B$.
 \end{enumerate}
 Notice that it follows from $(ii)$ that if $v \in B^{ann}$ and $w\in B $ is such that
 $\langle w ,v \rangle =0$, then $\langle \Dc_n^+w ,v \rangle=0$.

 \begin{proof} 

a) {\it Multiplicity 1}: (This is implied by
\Proposition{\ref{gen-branch}} but we give a separate proof) If
$N \subset N_\lambda $, where $N$ is a simple $\Dc_{n}$-submodule, by
\Corollary{\ref{cor:d0eqconcrete}} and \Proposition{\ref{symm-lie}}
$\Ann_{\Dc_n^-}(N)$ is a simple module over the commutative ring
$R_n=k[\nabla_1, \ldots , \nabla_{n}]^{S_n} $, so that by
$\Ann_{\Dc_n^-}(N)= k v_n$ for some polynomial $v_n$, as described in
\Proposition{\ref{r-iso}} (2). We expand in the variable $x_n$
\begin{displaymath}
  v_n =  v_{n-1}x_n^a+ w  \quad  (w \text{ is of lower order in } x_n)
\end{displaymath}
where
\begin{displaymath}
v_n \in \Ann_{\Dc_n^-}(N_\lambda) = \{\mu \in N_\lambda\ \vert \ \sum_{i=1}^n x_i^k\partial_i^l (\mu)=0, 0 \leq k < l \leq n-1\}.
\end{displaymath}
\Prop{\ref{symm-lie}}. Therefore
$\sum_{i=1}^{n-1} x_i^k\partial_i^l (v_{n-1})=0$, when
$0 \leq k < l \leq n-2$, so that
$v_{n-1}\in \Ann_{\Dc_{n-1}^-}(N_\lambda)= k v_\lambda$. We can therefore assume,
after multiplying $v_{n}$ by a constant, that $v_{n-1}= v_\lambda$. Assume
that $N' \subset N_\lambda $ and $N' \cong N$. Again
$\Ann_{\Dc_n^-}(N')= k v_n'$. Then clearly $\deg v_n = \deg v_n'$,
and after multiplying by a constant  we get the expansion
\begin{displaymath} 
v_n' =  x_n^a v_\lambda + w'.
\end{displaymath}
We assert  that $v_n= v_n'$.  
Assuming the contrary,
\begin{displaymath}
0 \neq   v_n- v_n' = w-w' = c x_n^b v_\lambda +  w", \quad  \deg (v_n- v_n' )=  a + \deg s_Q
\end{displaymath}
But $b <a$, which results in the  contradiction $\deg (v_n - v_n') < a + \deg
v_\lambda $.  
 
b) {\it Existence of submodules}:  
We have a natural exhaustive filtration by $\Dc_n$-submodules
\begin{displaymath}
  N_0= \Dc_n v_\lambda \subset \cdots \subset N_j= \sum_{i \leq j} \Dc_n x_n^iv_\lambda \subset \cdots \subset N_\lambda
\end{displaymath}
Clearly $N_0$ is a simple $\Dc_n$-module, so assume that $a >0$.
We assert:
\begin{align*}
  a&\in \lambda  \iff \\
&\text{ there
    exists  } v_a = x_n^a v_\lambda + w_a \in \Ann_{\Dc_n^-}(N_a) \text{ and
  }  k v_a \text{ is a simple } R_n-\text{module}.
\end{align*}

$\Leftarrow$: There exists a partition $P$ of $[n-1]$ such that
$\lambda^P = \lambda$ and $k v_\lambda \cong ks_P$ as
$R_{n-1}$-module. Then if $x^\alpha$ is a monomial term in $v_\lambda$
it follows that
$ \lambda^\alpha = \lambda^c \vdash n-1$, where $\lambda^c$ is the
conjugate of $\lambda$ (see \Lemma{\ref{keypartition}}). If there
exists a vector $v_a$ as stated, so that $x^\beta =x_n^a x^\alpha$ is
a monomial term in $v_a$, then $\lambda^\beta \vdash n$. This implies
that $a= (\lambda^\alpha)^c(i)= \lambda (i)$ for some index $i$. (We
can also say that $a$ is addable to $\alpha$ if
$a\in \lambda = (\lambda^\alpha)^c$.)

$\Rightarrow$: We can assume that $v_\lambda = s_Q$ for some partition
$Q= \{Q_i\}$ of $[n-1]$ such that $\lambda_Q= \lambda$. If $a= \lambda_i =|Q_i|$ we let
$P_j= Q_j$ when $j\neq i$, $P_i = Q_i \cup \{n \}$, so that $\{P_j\}$
is a partition of $[n]$. Then
\begin{displaymath}
  s_P = x_n^a s_Q + s'_P, \quad \deg_{x_n}s'_P < a.
\end{displaymath}
Notice that $s_P$ is not a semi-invariant of the Young subgroup of
$Q$, so that $s_P \not \in N_\lambda$. Since $\{x^\alpha\}$ is an
orthogonal basis for $B$ we have $<s_P', x_n^a s_Q>=0$ and hence
$<s_P, x_n^a s_Q > = <x_n^as_Q, x_n^as_Q> \neq 0$, and since $s_P$ is
a homogeneous minimal degree semi-invariant of a parabolic subgroup of
$S_n$ it follows that $\Dc_n^- s_P=0$. Moreover,
$\deg s_P = a + \deg s_Q > \deg (x_n^i s_Q)= i + \deg s_Q$, when
$i<a$, so that
\begin{displaymath}
  <s_P, \Dc_n x_n^i s_Q> = <\Dc_n s_P, x_n^is_Q>= <(\Dc_n^+ +  \Dc_n^0) s_P, x_n^i s_Q> =0,
\end{displaymath}
i.e. $s_P \perp N_{a-1}$. Therefore $x^a s_Q \not \in N_{a-1}$, hence
there exists elements of the form $v_b$ in
$\Ann_{\Dc_n^-}(N_a)\neq 0 $ where $b \geq a$. Now if
$v \in \Ann_{\Dc_n^-}(N_a)$ is homogeneous, then
$\deg P v \geq \deg v$ when $P\in \Dc_n$; since moreover
$N_a = \Dc_n \Ann_{\Dc_n^-}(N_a) $ it follows that there exists an
element in $\Ann_{\Dc_n^-}(N_a)$ of the form $v_a$ such that
$k v_n$ is $R_n$-simple.

It follows from a) and b) that
\begin{displaymath}
  N_\lambda = \bigoplus_{a\in \lambda} \Dc_n v_a.
\end{displaymath}

It remains to see that $\deg v_\lambda = \sum_i i \lambda (i) -l$, which
we prove by induction in $n-1$. The vector $v_\lambda$ arises from
some $v_{\lambda'} \in \Ann_{\Dc_{n-2}^-}(B)$ where
$\lambda' \vdash n-2$, in the form
$v_\lambda = x_{n-1}^{a'} v_{\lambda'}+ \cdots $. By induction
\begin{displaymath}
  \deg v_{\lambda'}=  \sum  i \lambda'(i) -l', \quad l'=\text{number of
    $i$ such that } \lambda'(i)\neq 0.
\end{displaymath} 
This implies that
\begin{displaymath}
  \deg v_\lambda = a' + \deg (v_{\lambda'}) = a' + \sum i\lambda' (i) -l'=  \sum \lambda (i) -l.
\end{displaymath}
since $a'= \lambda'(j)$ for some $j$, so that
$\lambda(j)= \lambda'(j)+1$, $\lambda(i)= \lambda'(i)$, $i\neq j$
(treat the cases $a'=0$ and $a'\neq 0$ separately).
\end{proof}

\subsection{The branching graph}\label{branching-graph}
Let us start with the general situation of inclusions of graded rings of the type
in \Section{\ref{abstract-eq}}
\begin{displaymath}
  \Dc_n \subset \Dc_{n-1}\subset \cdots \subset \Dc_2 \subset \Dc_1,
\end{displaymath}
so that in particular $\nabla \in \Dc_i$ for all $i$. Let $M$ be a simple $\Dc_1$-module such that
its restriction to $\Dc_i$-module is semisimple and $\Dc_i \Ann_{\Dc_i^-}(M)=M$, $i=1, \dots , n$
(see \Theorem{\ref{main}} and \Section{\ref{inv-rings}}). Let $\Cc_i$ be the set of isomorphism
classes of simple $\Dc_i$-submodules of $M$.  The branching graph $\Bc(M)$ of $M$ (or oriented
Bratteli diagram) is defined as follows. Its set of vertices is $\cup_{n\geq 1} \Cc_i$ and there are
$\dim_{\Cb}Hom_{\Dc_i}(N_\lambda,\reso_{\Dc_{i-1}}^{\Dc_i} N_{\mu}) $ directed edges from the vertex
$\mu\in \Cc_{i-1}$ to the vertex $\lambda \in \Cc_i$ (where $N_\lambda, N_\mu$ are representative
modules for $\lambda$ and $\mu$), and there are no other edges. Let us agree to say that a vertex
$\lambda $ in $\Bc(M)$ has the level $i$ if $\lambda \in \Cc_i$, and write $|\lambda|=i$.  Write
$T>T'$ when $T$ and $T'$ are directed paths in $\Bc(M)$ with a common first vertex (which normally
is the root $\Cc_1$ of $\Bc(M)$; this is a singleton set) and the last vertex of $T'$ is joined by
an edge with the last vertex of $T$. It is a fundamental problem to give a combinatorial description
of the oriented rooted tree $\Bc(M)$, given parametrisations of the sets $\Cc_i$.

We will give such a description of $\Bc = \Bc(B)$ when $\Dc_i=\Dc_B^{S_i}$, so that by
\Theorem{\ref{branch-diff}} there is at most one edge between two vertices.  The Young graph
$\Yc$ is the oriented graph whose set of vertices is $\cup_{i=1}^n \Pc(i)$, where $\Pc(i)$ is the
set of partitions $\lambda \vdash i$.  There is an edge from the vertex $\lambda' \in \Pc(i-1) $ to
$\lambda \in \Pc (i)$ if the Young diagram of $\lambda$ is obtained from $\lambda'$ by adding 1
addable box, and there are no other edges in $\Yc$. It is easy to see that the paths in $\Yc$ are in
correspondence with standard Young tableaux.

\begin{proposition}\label{young-branch}
  $\Bc $ is isomorphic to $ \Yc$.
\end{proposition}
\begin{proof}
  Since $\Cc_i$ is parametrised by the set of partitions of the integer $i$ \Prop{\ref{prop:specht}}
  it is clear that the cardinality of the set of vertices in $\Bc$ agrees with that of $\Yc$. That
  the edges agree follow from \Theorem{\ref{branch2}}.
\end{proof}

In fact, the branching graph $\Bc$ is isomorphic to the branching graph $\Cc$ of the sequence of
group inclusions
$S_1 \subset S_2 \subset \cdots \subset S_{i-1}\subset S_i\subset \cdots \subset S_n$, described in
detail in \cite{okounkov-vershik} (see also \cite{kleshchev-linearrep}). The vertex set of $\Cc$ is
$\cup_{i \geq 1} \hat S_i$ and there is an edge from the vertex $\mu\in \hat S_{i-1} $ to the vertex
$\lambda \in \hat S_{i}$ if $\dim_{\Cb } Hom_{S_i}(\reso^{S_{i-1}}_{S_i}V^\lambda, V^\mu)=1$
($\leq 1$ by \Proposition{\ref{gen-branch}}), where $\reso^{S_{i-1}}_{S_i}V^\lambda$ is the
restriction of a representative $V^\lambda$ of $\lambda$ to a representation of $S_{i-1}$, and there
are no other edges in $\Cc$.  It follows from \Proposition{\ref{ind-res}} and Frobenius reciprocity
that the oriented graphs $\Bc$ and $\Cc$ are isomorphic.
\begin{remark}\label{remark-okounkov-vershik}
  In \cite[Th. 6.7]{okounkov-vershik} it is proven that $\Cc $ is
  isomorphic to $ \Yc$, which together with
  \Proposition{\ref{ind-res}} also implies
  \Proposition{\ref{young-branch}} (and {\it vice versa},
  \Proposition{\ref{young-branch}} implies $\Yc\cong \Cc$). Notice
  that in our setup, where first the branch rule in
  \Theorem{\ref{branch2}} is proven rather explicitly, one immediately
  gets that $\Cc $ is isomorphic to $ \Yc$. The proof of [loc cit] is
  based on an identification of paths in $\Cc$ with paths in $\Yc$,
  using the notion of weights of a GZ-algebra as a bridge, so that the
  complete branching rule in \Theorem{\ref{branch2}} (transcribed to
  groups) is achieved at the same time as [loc cit] is established.
\end{remark}
Below we will make no distinction between a path in the branching graph and a path in the Young
graph, so that paths in $\Bc$ of length $n$ are identified with standard tableaux of size $n$.

\subsection{The canonical basis of $B^{ann}$}\label{canonical-section} Let as above
$\Dc_i = \Dc_B^{S_i}= \Dc(V_i)^{S_i}\otimes_{\Cb}\Dc(V_i') $, where $V_i= \sum_{j=1}^i \Cb x_j$ and
$V'_i = \sum_{j=i+1}^n \Cb x_j$. We have
$\Rc_i = \Dc_i^0/ (\Dc_i^0\cap \Dc_i \Dc_i^- )\cong \Cb[\nabla_1, \ldots , \nabla_i]^{S_i}$, so that
$B_i^{ann}= \Ann_{\Dc_i^-}(B)$ is a finite-dimensional semisimple $\Rc_i$-module.  Let $i>1$ be an
integer and $ C_{i-1}= \{v^{i-1}_j\}_j$ be a basis of $B_{i-1}^{ann}$ which is compatible with its
isotypic decomposition as $R_{i-1}$-module.  Now the vector space
$\Ann_{\Dc_{i}^-}(\Dc_{i-1}v^{i-1}_j)$ is a multiplicity free $\Rc_i$-module
\Props{\ref{gen-branch}}{\ref{main}}, so we can select a basis $C_{i}^j = \{ v^{i-1}_{jk}\}_k$ such
that the $\Rc_i$-modules $\Cb v^{i-1}_{jk} $ are simple and mutually non-isomorphic for different
$k$; hence this basis is unique up to scalars.  Then $C_i = \cup_j C_{i}^j $ is a canonical basis of
$B_i^{ann}$ given the basis of $B_{i-1}^{ann}$. Since $\dim_{\Cb} B_1^{ann}=1 $ it follows by
iteration that we get a basis $C= \{v_T\}$ of $B_n^{ann}$ which is unique up to scalars, and where
the basis elements $v_T$ are indexed by paths $T$ of length $n$ in the branching graph $\Bc$.  By
\Proposition{\ref{young-branch}} the set of paths $T$ from the root of the graph $\Bc$ with the same
endpoint $\lambda \in \Pc(n)$ can be parametrised by the set of standard tableaux of shape
$s(T)=\lambda$. We have a decomposition

\begin{displaymath}
B^{ann}= \bigoplus_{\lambda \in \Pc(n)} V_\lambda, \quad   V_\lambda = \bigoplus_{s(T)= \lambda} \Cb v_T,
\end{displaymath}
where $V_\lambda$ are the $\Rc_n$-isotypical components of $B^{ann}$, and they form also the
irreducible representations of $S_n$.

\begin{example}\label{example-3}
  Let $n=3$.  The decomposition into isotypical components in terms of a canonical (Young) basis is
  of the form
\begin{align*}
  B_2^{ann} &= V_{\{1,1\}} \oplus V_{\{2\}} = \Cb \oplus \Cb (x_1-x_2),\\     B_3^{ann} &=
  V_{\{1,1,1\}}\oplus  V_{\{2,1\}} \oplus V_{\{3\}}, \quad   V_{\{1,1,1\}} =  \Cb, \\  V_{\{2,1\}} &=   \Cb (x_1+ x_2 - 2 x_3)\oplus  \Cb (x_1-x_2), \quad  
  V_{\{3\}} =  \Cb (x_1-x_2)(x_1-x_3)(x_2- x_3).
\end{align*}
\end{example}

Young bases of an irreducible $S_n$-representation $V_\lambda$ were first introduced in
\cite{young-collected}; see also \cite{james-kerber}*{\S 3.2}. The construction of a basis of a
representation of a group by using nested sequences of subgroups, so that the restriction in each
step is multiplicity free, was developed for the groups $\SOo(n)$ and $\Uo(n)$ in
\citelist{\cite{gelfand-zetlin:unimod}\cite{gelfand-zetlin:orto}}, and one refers often therefore to
Gelfand-Zetlin bases.  In \cite{murphy:newconstr} and \cite{jucys}, independently (see also
\cite{okounkov-vershik}), Young bases were constructed using the algebra $A(n)$ that is generated by
the so-called Jucys-Murphy elements $\{L_i\}^n_{i=2}\subset \Cb[S_n]$, where
$L_i = \sum_{j=1}^{i-1} (j \ i)$ ($(j \ i)$ is a transposition).  Such bases
$\{w_T\}_{T\in \Sc_\lambda}\subset V_\lambda$ are indexed by the set $\Sc_\lambda$ of standard
tableux of shape $\lambda\vdash n$, and are characterised by the fact that
$V_\lambda = \oplus \Cb w_T$ where each $\Cb w_T$ forms a simple $A(n)$-module, and
$\Cb w_{T_1} \not \cong \Cb w_{T_2}$ when $T_1\neq T_2$;  they are uniquely determined by this condition
up to scalars.  Since $B^{ann}= \oplus V_\lambda$ is multiplicity free and hence canonically
decomposed into irreducible $S_n$-representations \Th{\ref{thm:rcomm}} we also get a unique (up to
scalars) basis of $B^{ann}$ by taking the union of Young bases of the irreducibles $V_\lambda$.

Let $\{v_T\}_{T\in \Sc}$ be the canonical basis  of $B^{ann}$, where $\Sc$  is the set of standard
tableaux of size $n$ (or set of  paths of length $n$ in the branching graph). 

\begin{thm}\label{can-base}
  A canonical basis of $B^{ann}$ is the same as a Young basis of the $S_n$-representation $B^{ann}$.
  In particular, the canonical basis vectors are common eigen-vectors of the Jucys-Murphy elements
  in $\Cb[S_n]$.
\end{thm}
\begin{proof}
  It suffices, by the above description of Young bases to see that $A(n) v_T= \Cb v_T$, which we
  prove by induction over $n$. It is evidently true when $n=1$, so assume that
  $A(n-1) v_{T'}= \Cb v_{T'}$ when $T'$ is a path of length $n-1$.  We have to prove that if
  $\Cb v_T $ is a simple $R_n$-submodule of $\Ann_{\Dc_n^-}(\Dc_{n-1}v_{T'})$, then
  $A(n)v_T= \Cb v_T$. Since $[A(n-1), \Dc_{n-1}]= 0$ and $[A(n-1), \Dc_n]=0$ it follows by induction
  that $A(n-1) v_T \subset \Ann_{\Dc_n^-}(\Dc_{n-1} v_{T'})$.  Therefore $\Cb v_T$ and $\Cb L_iv_T$,
  $i\leq n-1$, are $R_n$-submodules of $\Ann_{\Dc_n^-}(\Dc_{n-1}v_{T'})$ of equal support, since
  $[L_i, R_n]=0$. By the branch rule it follows that the $R_n$-module
  $\Ann_{\Dc_n^-}(\Dc_{n-1}v_{T'})$ is multiplicity free, which implies that $L_i v_T \in \Cb v_T$;
  hence $A(n-1)v_T = \Cb v_T$.  We can write $L_n= Z_n + Z_{n-1}$ where $Z_n$ and $Z_{n-1}$ belong
  to the center of $\Cb[S_n]$ and $\Cb[S_{n-1}]$, respectively. Since $A(n-1) $ is a maximal
  commutative subalgebra of $\Cb[S_{n-1}]$ \cite{diaconis-greene}, it follows that
  $Z_{n-1}\in A(n-1)$, so that $Z_{n-1}v_T \in \Cb v_T$. The element $Z_n$ acts by a scalar on the
  simple $\Dc_n[S_n]$-module $B_i$ in \Proposition{\ref{thm:montgomery}}.  Since the vector $v_T$
  generates a simple $\Dc_n$-module it is contained in precisely one isotypical component
  $B_i$. Hence $Z_nv_T \in \Cb v_T$. This completes the proof that $A(n)v_T = \Cb v_T $.
\end{proof}

Consider now  an  expansion
\begin{displaymath}
  v_T = c x^{\alpha_T} + \text{l.o.}, 
\end{displaymath}
where we use the reverse lexicographic ordering of the multiindices, so that $\alpha > \beta$ if for
some integer $1 \leq m\leq n$ we have $\alpha (i) \geq \beta (i) $ for $m \leq i \leq n$ and
$\alpha(m) > \beta(m)$. Thus ``l.o.'' signifies a sum of monomial terms $x^\beta$ such that
$\alpha_T> \beta$.

We consider also the support of the simple $R_n$-module $\Cb v_T$. The support of any
$R_n$-submodule of $B$ is a subset of $ \Zb^n\subset \Cb^n= \mSpec R_n$, where $\Cb^n$ is identified
with $Hom_{\Cb}(\Cb t_1+ \cdots + \Cb t_n, \Cb)$ and $t_i = \sum_{j=1}^n \nabla_j^i$.  We embed
$\Zb^i \subset \Zb^n$ by
$(\gamma_1, \ldots , \gamma_i)\mapsto (\gamma_1, \ldots , \gamma_i, 0, \ldots, 0) $. If $T$ is a
path of length $n$ we let $\{\gamma_T\} \subset \Zb^n$ be the support of the $R_n$-module $\Cb v_T$.

Let $F_{n+1}(x_1, \ldots , x_n)$ be the polynomial such that
\begin{displaymath}
  h_{n+1}(x_1, \ldots , x_n) = F_{n+1}(h_1(x), \ldots , h_n(x)),
\end{displaymath}
where the power sums $h_i$ are defined
in (\ref{gen-proc}).

For a directed path $T$ of length $n$ in $\Bc$ we let $\lambda_T \vdash n$ be the integer partition
that corresponds to the endpoint of $T$ at the level $n$. 
\begin{proposition}\label{lem-spec} Assume that $T > T'$ so that $T'$ is a path of length $n-1$.
  Given $\alpha_{T'}$ and $\gamma_{T'}$, the set of possible $\alpha_T$ and $\gamma_T$ is
  \begin{enumerate}
  \item $\alpha_T= \alpha_{T'} + (0,\ldots , a )$
  \item
    \begin{displaymath}
\gamma_T(i) =
\begin{cases}
  \gamma_{T'}(i)+a^i,\quad   1\leq i < n, \\
F_{n}(\gamma_{T'}(1),\ldots ,  \gamma_{T'}(n-1)) + a^n,\quad  i=n,  
\end{cases}
\end{displaymath}
  \end{enumerate}
where $a\in \lambda_{T'}\vdash n-1$, or
    $a=0$.
\end{proposition}
\begin{proof} Let $N= \Dc_{n-1} v_{T'}$ and $M= \Dc_n v_T$, where
  $v_{T'}\in \Ann_{\Dc_{n-1}^-}(N)$ and $v_T \in \Ann_{\Dc_n^-}(N)$.  As in the proof of
\Theorem{\ref{branch-diff}} we have $\deg M = \deg N + a$, where $v_T=cx_n^a v_{T'}+ \text{l.o.}$,
and any simple submodule of $N$ is determined by $\deg M$, and hence by the integer $a$.  By
\Theorem{\ref{branch2}} the possible $a$ are determined by the branch rule, which means that
$a\in \lambda_{T'}$ or $a=0$.

  (2): Put $t_i^{(n)}= \sum_{j=1}^n\nabla_j^i $, so that $t_i^{(n)} = t_i^{(n-1)} + \nabla_n^i$.
  Since $v_T = c x_n^a v_{T'}+ \text{l.o.}$ is an eigenvector of $t_i^{(n)} $ and $v_{T'}$ is an
  eigenvector of $t_i^{(n)}$, it follows that when $i <n$ , then
\begin{displaymath}
  t_i^{(n)} v_T = (\gamma_{T'}(i) + a^i) v_T. 
\end{displaymath}
and
\begin{displaymath}
  t_n^{(n)}v_T = (F_{n}(t_1^{n-1}, \ldots , t_{n-1}^{(n-1)} ) + \nabla_n^n)v_T = (F_{n}(\gamma_{T'}(1),\ldots ,  \gamma_{T'}(n-1)) + a^n)v_T.
\end{displaymath}
\end{proof}
We can compare with the spectrum $\ell_T: [n]\to \Nb$ of the
Jucys-Murphy elements $L_i$, $L_i v_T = \ell_T(i)v_T$, given by
$\ell_T(i)= s-r$, where $r$ and $s$ are the row and column number of
the box containing $i$ \cite{murphy:newconstr} (the number $s-r$ is
called the ``class'', ``content'', or ``residue'' of the box of $i$).
We thus have:
\begin{displaymath}
  \alpha_T(i)= r-1  \quad \text{and}\quad   \ell_T(i)=s-r.
\end{displaymath}
By \Proposition{\ref{lem-spec}} one can read off $\alpha_T$ and $\gamma_T$ (and $\ell_T$) from the
standard tableau $T$.  The value $\alpha_T(i)$ equals the number of boxes above the box of $i$ in
$T$, where the box of $i$ is inserted in $T'$ at an addable position.  It is therefore clear that
the map
\begin{displaymath}
  \Bc \to (\alpha : [n]\to \Nb): T\to \alpha_T
\end{displaymath}
is injective. One can similarly recover $T$ from $\gamma_T$ by considering the successive removing
of the boxes of $n, n-1, \dots , 2$ from $T$. We can conclude that paths in the branching graph
$\Bc$ are determined by the multi-index $\alpha_T$ (i.e. eigenvalues of the $x_i \partial_i$) or the
eigen-values of the elements $t_i$. One can compare to the description of the branching graph $\Cc$
in \cite{okounkov-vershik} in terms of the spectrum of the Jucys-Murphy elements $L_i$.

The relation $v_T= c x_n^a v_{T'}+ \text{l.o.}$ also makes it straightforward to recover the
canonical basis at one level from a higher level basis.
\begin{proposition}\label{decompose} Embed $S_k$ in $S_n$ such that $S_k$ fixes the variables
  $x_{k+1}, \ldots , x_n$.  Select for each path $T'$ of length $k$ a path $T$ of length $n$ that
  starts with $T'$, and let $\{v_T\}$ be the corresponding subset of a canonical basis of
  $\Cb[x_1, \dots , x_n]^{ann}$.  Expand
  \begin{displaymath}
    v_T= x_n^{a_n}  x^{a_{n-1}}_{n-1} \cdots  x^{a_{k+1}}_{k+1} v_{T'}  + \text{l.o.},
  \end{displaymath}
  where l.o. signifies terms of lower order than $(a_{k+1}, \ldots , a_n)$ in the reverse
  lexicographic ordering of the set of multiindices $[n-k]\to \Nb$.  Then $\{v_{T'}\}$ is a
  canonical basis of $\Cb[x_1, \ldots , x_k]^{ann}$. Moreover, the exponent $a_i$ equals the number
  of boxes in $T$ above the box of the integer $i$.
\end{proposition}

We have already motivated the last assertion in
\Proposition{\ref{decompose}}. The proof of the remaining part  follows from the following
lemma.

\begin{lemma} 
  Let $p $ be a homogeneous polynomial in $B$ and $1\leq i < n$ an
  integer. Make the expansion
\begin{displaymath} p= \sum_{\beta_i} p_{\beta_i}x^{\beta_i}, 
\end{displaymath} 
where $p_{\beta_i}\in \Cb[x_1, \ldots , x_i]$, $\beta_i: [n-i]\to \Nb$, and
$x^{\beta_i} = \prod_{j=1}^{n-i} x_{i+j}^{\beta_i(j)}$.
\begin{enumerate}
\item If $R_n\cdot p = \Cb p$, then $R_i\cdot p_{\beta_i} = \Cb p_{\beta_i}$.
\item Assume that $\beta_a : [n-i]\to \Nb $ is maximal in the reverse lexicographic order of
  multiindices such that $p_\beta\neq 0$ in the above expansion.  If $p \in \Ann_{\Dc_n^-}( B)$,
  then $p_{\beta_a} \in \Ann_{\Dc_i^-} (B) $.
\end{enumerate}
\end{lemma}
\begin{proof} By iteration it suffices to prove (1) and (2) when
  $i=n-1$.  

  (1): We have an index $\beta_{n-1}: [1]\to \Nb$ determined by a
  single integer $\beta_{n-1}(1)= j$. Consider the expansions
  \begin{displaymath}
    p = \sum c_\alpha x^\alpha = \sum_{j=1}^a p_{\beta_{n-1}} x_n^j=   \sum_{j=1}^a(\sum_{\gamma_j} d_{\gamma_j} x^{\gamma_j}) x_n^j.
  \end{displaymath}
  Since $R_n\cdot p_n = \Cb p_n$ the multi-indices $\alpha $ with $c_\alpha\neq 0$ are of the form
  $\sigma \cdot \beta$ for some fixed $\beta$ with $\sigma$ running over some elements in
  $S_n$. This implies that the multi-indices $\gamma_j :[n-1]\to \Nb$ are of the form
  $\sigma' \cdot \beta'_j$ for some fixed $\beta'_j$ and $\sigma'$ runs over elements in $S_{n-1}$.
  This implies that $R_{n-1} p_{\beta_{n-1}} = \Cb p_{\beta_{n-1}}$.  

  (2): Since
  $p^{(n)}_{k,l} =  p^{(n-1)}_{k,l} + x^k_n \partial_n^l\in \Dc_n^-
  $ we have
\begin{displaymath}
  0= p^{(n)}_{k,l} (p)=  p^{(n-1)}_{k,l} (p_{\beta_a})x^{\beta_a} + p_{\beta_a} x^k_n \partial_n^l (x^{\beta_a})  +
  (\text{l.o. in  }  x')
  = p^{(n-1)}_{k,l} (p_{\beta_a})x^{\beta_a} + (\text{l.o. in } x'),
\end{displaymath}
implying that $p^{(n-1)}_{k,l} (p_{\beta_a})=0$. By  \Lemma{\ref{symm-lie}}
\begin{displaymath}
\Dc_{n-1}^- =\sum_{0\leq k < l < n-1}  \Dc_{n-1} p^{(n-1)} _{k,l},
      \end{displaymath}
implying that  $\Dc^-_{n-1}\cdot p_{\beta_a} =0$.
\end{proof}

\begin{proposition} Let $\Sc$ be the set of standard tableaux of size $n$ and $\{v_T\}_{T\in \Sc}$
  be a canonical (Young) basis of $B^{ann}$. This induces a decomposition of $B$ as $\Dc_n$-module
  \begin{displaymath}
    B= \bigoplus_{T\in \Sc} N_T,
  \end{displaymath}
  where the $\Dc_n$-module $N_T= \Dc_n v_T$ is simple. The $A$-module $N_T$ is free  of rank equal
  to the number of standard tableux of shape $s(T)= \lambda$.
\end{proposition}
\begin{proof}
  By \Proposition{\ref{thm:montgomery}} $B= \oplus B_\lambda$, where the $B_\lambda$ are isotypical
  components of $B$ and a subset $\Bc_\lambda \subset \{v_T\}_{T\in \Sc}$ gives a basis of
  $B_\lambda^{ann}= V_\lambda$.  Since $\Dc_n v_T$ is a simple $\Dc_n$-module \Th{\ref{main}} it
  follows that $B_\lambda= \oplus_{v_T\in \Bc_i } \Dc_n v_T$. Since $B$ is free over $A$, it follows
  by semisimplicity that $\Dc_n v_T$ is also free.  It is well-known that as $A[G]$-module
  $B\cong A[G]$, the regular representation, which, by \Proposition{\ref{thm:montgomery}}, implies
  that $\rank N_T = \dim_{\Cb} V_\lambda$, where $V_\lambda$ has a basis $\{v_T\}$ indexed by all
  paths $T$ that end at $\lambda$; and the number of such paths equals the number of standard
  tableux of shape $\lambda$.
\end{proof}
By \Lemma{\ref{gen-R}}, (2), and \Lemma{\ref{symm-lie}}, (2), (with $m=1$) we have
$\bar \af^- \subset \Dc_n^- = \sum_{0\leq k < l < n} \Dc_n\bar \af^-$, and $t_i v_T = \gamma_T(t_i)v_T$.
Hence
\begin{displaymath}
  \sum_{0\leq k < l < n} \Dc_n(\sum_{i=1 }^n x_i^{k} \partial_i^l) + \sum_{i=1}^n \Dc_n (t_i -
  \gamma_T(t_i)) \subset \Ann_{\Dc_n}(v_T).
\end{displaymath}
This is however not an equality, and it would be interesting to find a complete description of
$\Ann_{\Dc_n}(v_T)$. Since $N_T$ is a free module over $A$ of finite rank it follows that for fixed
$k>l$ the set $\{p_{k,l}^iv_T\}_{i=0}^r$ is not linearly independent over $A$ for sufficiently high
$r$.  Let for $0 < l < k < n$
\begin{displaymath}
  P_{k,l} = p_{k,l}^r + a_1 p_{k,l}^{r-1}+ \cdots + a_0 \in \Dc_n^+ \cap \Ann_{\Dc_n}(v_T),
\end{displaymath}
where $a_i \in A$, and $r$ is of minimal degree $r\geq 1$. One may ask if
 \begin{displaymath}
   \Ann_{\Dc_n}(v_T) =    \sum_{0\leq k < l <  n} \Dc_n(\sum_{i=1 }^n x_i^{k} \partial_i^l) + \sum_{i=1}^n \Dc_n (t_i -
  \gamma_T(t_i))  + \sum_{0 < l <k < n} \Dc_n P_{k,l},
 \end{displaymath}
 and also what is the precise form of the $P_{k,l}$?

 \begin{remark} Murphy \cite{murphy:newconstr} expressed idempotents $E_T$ in terms of the
   Jucys-Murphy elements $L_i$, so that if $\{s_T\}$ is the standard Specht basis and one puts
   $v_T = E_Ts_T$, then $\{v_T\}$ is a Young basis, where the two bases of $B^{ann}$ are transformed
   inte one another by a unimodular triangular matrix.  Still, it would be desirable to find a
   concrete decomposition of the $\Rc_i$-module $\Ann_{\Dc_i^-}(\Dc_{i-1} v_{T_{i-1}})$, perhaps by
   refining the proof of \Theorem{\ref{branch2}}, and thus getting a direct construction of
   $\{v_{T_i}\}_{T_i\in \Sc_i}$ from $\{v_{T_{i-1}}\}_{T_{i-1}\in \Sc_{i-1}}$, where $\Sc_i$ is the
   set of standard tableaux of size $i$. The complexity of finding a basis of
   $\Ann_{\Dc_i^-}(\Dc_{i-1} v_{T_{i-1}})$ using a direct linear algebra approach is determined by
   the degrees $\deg v_{T_{i-1}}= |\alpha_{T_{i-1}}|$ and $\deg v_{T_i} = |\alpha_{T_i}|$; see the
   example below.
\end{remark}

Let us take \Example{\ref{example-3}} one step further.
\begin{example}\label{level4} Consider the tableau \ytableausetup{centertableaux} T= \begin{ytableau}
    1 & 3  \\
    2 & 4  \\
\end{ytableau}
which is formed from \ytableausetup{centertableaux} T'= \begin{ytableau}
    1 & 3  \\
    2   \\
\end{ytableau} 
by adding the element 4 to the subset $\{3 \}\subset [3]$.  We have $\alpha_{T'} = (0,1,0)$ and
$\gamma_{T'}= (1,1,1)$. By \Lemma{\ref{lem-spec}} $\alpha_T = (0,1,0, 1)$ and, since 
$F_4(y_1, y_2, y_3) = \frac 16 y_1^4 - y_1^2y_2+ \frac 12 y_2^2 + \frac 43 y_1 y_3 $, so that
$F_4(1,1,1)=1$,  
\begin{displaymath}
\gamma_T = (1,1,1, 0) + (1,1,1,0) + (0,0,0, 1+F_4(1,1,1))= (2,2,2,2).
\end{displaymath}
We also can see that $\ell_T= (0,-1,1,0)$.  The canonical basis $v_{T'}= 2x_3 - (x_1+ x_2)$ (we
ignore scalars in $v_{T'}$ and $v_T$ ).  Put $p_{k,l}^{(m)}= \sum_{i=1}^m x_i^k\partial_i^l$.  A
moments reflections shows that the following Ansatz suffices
\begin{displaymath}
v_T = x_4 v_{T'} + v_T(0) = (x_4 +c_1p^{(3)}_{1,0} + c_2 p^{(3)}_{2,1}) v_{T'},
\end{displaymath}
where $v_T(0)$ is independent of $x_4$.  The conditions are
\begin{align*}
  p^{(4)}_{0,1} v_T &= [p^{(4)}_{0,1}, x_4 + c_1p^{(3)}_{1,0} + c_2 p^{(3)}_{2,1}] v_{T'} = (1+ 3c_1
                      + 2 c_2 p^{(3)}_{1,1})  v_{T'}= (1+ 3c_1 + 2 c_2) v_{T'}=0\\
  p^{(4)}_{1,2} v_T &= [p^{(4)}_{1,2}, x_4 + c_1p^{(3)}_{1,0} + c_2 p^{(3)}_{2,1}]v_{T'}  = (2x_4\partial_4  +2c_1p^{(3)}_{1,1} +
                      c_2(3p^{(3)}_{2,2} + 2p^{(3)}_{1,1}) )v_{T'} \\ 
                    &=  ( 2c_1 + 2 c_2)v_{T'}=0,
\end{align*}
which give
\begin{displaymath}
  v_T = (x_4 + \frac 13 (p^{(3)}_{1,0}- p^{(3)}_{2,1})) v_{T'} =2 x_4 x_3 -x_4
  x_2-x_4x_1-x_3x_2-x_3x_1+2x_2x_1,
\end{displaymath}
where the expansion is in reverse lexicographic order. To exemplify \Proposition{\ref{decompose}},
we have
\begin{displaymath}
v_{T} = x_4v_{T' } + v_T(0)= x_4(x_3 v_{T''} + v_{T''}(0)) +v_T(0) = x_4 x_3v_{T''} + w,
\end{displaymath}
where $T> T'>T''$, so that for $k=3$, $v_{T'} = 2x_3 -(x_1+x_2)$, and for $k=2$, $v_{T''}=2$
(again, we do not keep track of scalars).

\end{example}

\section{Cyclic and Dihedral groups} \label{cycl-dihedral}
We will in this section study invariants of the dihedral group $G= G(e,e,2) = A(e,e,2)\rtimes S_2$,
already discussed in (\ref{imprimitive}), and its normal cyclic subgroup
$C_e = A(e,e,2) = \langle1,\rho,\rho^2,\ldots ,\rho^{e-1}\rangle$, where the action of $G$ is
\begin{displaymath}
\rho (x_1)=\epsilon x_1, \quad \rho (x_2)=\epsilon^{-1} x_2,  \quad \text{and}\quad  \sigma(x_1)= x_2
\end{displaymath}
on a basis $\{x_1, x_2\}$ of $V= \Cb^2$, where $\epsilon=\exp(2\pi i/e)$. We will in particular
describe $\Rc$ and its module $B^{ann} $ for these two groups, as well as describe the MLS-restriction
functor induced by $C_e\subset G$. The results on invariants and semi-invariants are of course
classical and surely date back to Gordan, but we include the computations for completeness.  The
invariant rings are
\begin{displaymath}
A_1=B^{C_e}=\Cb [x_1^e, x_2^e,x_1x_2]\quad \text{ and } \quad
A_2=B^G=\Cb [x_1^e+x_2^e,x_1x_2]
\end{displaymath}
(see the proof of \Lemma{\ref{power-inv}}).  Notice that the cyclic
group $C_e$ is not generated by complex reflexions of $V$, so that
$A_1$ is not a polynomial ring, in contrast to $A_2$. The ring $A_1$
corresponds to one of the two infinite series of Kleinian surface
singularities; it may be of interest to also study the other infinite
series of $B^{\bar G}$, where $\bar G$ is the binary dihedral
extension of $G$, see \cite{dolgachev-mackay}.

The invariant differential operators are denoted $\Dc_1= \Dc_B^{C_e}$ and $\Dc_2 = \Dc_B^G$, so
that $\Dc_2 \subset \Dc_1$.

\subsection{The cyclic group $C_e$}\label{cyclic-group}
 Since $C_e$ is abelian, each $C_e$-isotypic component of $B$ is by \Proposition{\ref{thm:montgomery}} a
 simple $\Dc_1$-module. Hence in order to find simple  $\Dc_1$-modules it suffices to find the semi-invariants of $C_e$. Let
$$\chi_i:C_e\to \Cb^*,\ \ \ 
 \chi_i(\rho)=\epsilon^i,\ \ \ 0\leq i<e,$$
 be the linear characters of $C_e$.

\begin{lemma}
\label{prop:cyclic} The isotypic component associated to $\chi_i\in \hat C_e$ is
\begin{displaymath}
N_i=A_1 x_1^i + A_1x_2^{e-i}.
\end{displaymath}
In particular,
\begin{displaymath}
B=\bigoplus_{0\leq i\leq e-1}N_i,
\end{displaymath}
is the decomposition of $B$ into simple  $\Dc_1$-modules.
\end{lemma}
Notice that $N_i$ is not free over $A_1$ even though $A_1$ and  $N_i$ are simple $\Dc_1$-modules. 
\begin{proof} Since $x_1x_2,x_1^e,x_2^e\in A_1$, any monomial may be written as
  $$a=bx^j_\delta,\ \delta=1,2, \ 0\leq j\leq e-1, b\in A_1.$$ Such a monomial belongs to the
  $\chi_i$-isotypic component if and only if $j= i$, when $\delta=1$, or $e-j=i$, when
  $\delta=2$. This implies that $N_i$ is the $\chi_i$-isotypical component of $B$.
\end{proof}

To determine $N_i^{ann}$, notice that if $i\neq e/2$ then the lowest non-zero degree component of
$N_i$ is spanned by either $x_1^i$, when $0\leq i<e/2$, or $x_2^{e-i}$, when $0\leq e-i<e/2$.  These
components are hence one-dimensional and will therefore coincide with $N_i^{ann}$ by
\Corollary{\ref{cor:d0eqconcrete} (3)}. This gives most of the following result:
\begin{proposition}
\label{prop:cyclic2}
\begin{displaymath}
N^{ann}_i=\begin{cases} \Cb x_1^i & 0\leq i\leq (e-1)/2,\\
\Cb x_1^{e_1}+ \Cb x_2^{e_1}, & i= e_1,\text{  when $e=2e_1$ is an even number,}\\
\Cb x_2^{e-i} &  (e-1)/2< i\leq e-1.
\end{cases}
\end{displaymath}

 \end{proposition}
 \begin{proof} It remains to treat $i=e_1$ when $e= 2e_1$ is even. Then the lowest non-zero
   component is $W=\Cb x_1^{e_1} + \Cb x_2^{e_1}$, and to see that $W=N^{ann}_{e_1}$ it suffices to
   see that $W$ is a simple $\Rc_1$-module. This follows since $\Dc_1^0$ contains the elements
   \begin{displaymath}
   \nabla_1,\nabla_2, (x_1\partial_2)^{e_1}, (x_2\partial_1)^{e_1},
 \end{displaymath}
 and these elements span the four-dimensional space $\End_{\Cb}(W)$.
\end{proof}

Next we will describe $\Rc_1= \Dc_1^0/\Dc_1^0 \cap (\Dc_1\cdot \Dc_1^-)$. By the just shown
existence of a simple $\Dc_1$-module $N$ such that $N^{ann}$ is not of dimension 1, we know by
\Theorem{\ref{thm:rcomm}} that, for even $e$, $\Rc_1$ is non-commutative. However, in the case of $e$
odd, $\Rc_1$ will be seen to be commutative.

\begin{proposition}
\label{prop:cyclic3}  Let $C_e$ be the cyclic group acting on  $\Cb^2 $ as above and  consider the canonical composed homomorphism 
 $$\Pi: \mathbb C[\nabla_1,\nabla_2] \to \Dc_1^0\to  \Rc_1 =\frac {\Dc_1^0}{\Dc_1^0\cap (\Dc_1 \cdot
   \Dc^-_1)}.$$ Then we have:
\begin{enumerate}
\item If $e$ is odd, then  $\Dc_1^0$ is generated as an algebra by 
$$\nabla_1,\nabla_2, (x_1\partial_2)^{e}, (x_2\partial_1)^{e},$$
$\pi$ is surjective, and hence $\Rc_1$ is a commutative ring.
 \item If $e= 2e_1$ is even,  $\Dc_1^0$ is generated as an algebra by the elements $$\nabla_1,\ \nabla_2, \ t_1=(x_1\partial_2)^{e_1}, \ t_2=(x_2\partial_1)^{e_1}.$$
Let $\Ec=\Imo (\Pi)$ be the algebra that is formed as the image of $\Cb [\nabla_1,\nabla_2]$.
Then $\Rc_1$ is generated by $1,t_1$ and $t_2$  as left (or right) module over $\Ec$.
\end{enumerate}
\end{proposition}

\begin{proof} We will use the method in (\ref{lemma:gr}), and therefore start by computing
  $\So(V\otimes_{\Cb} V^*)^{C_e}$.  Here $V\otimes_{\Cb} V^*= V_0 \oplus V_1 $ where
  $V_0 = \Cb \nabla_1 + \Cb \nabla_2$, $V_1 = \Cb b_1 + \Cb b_1$, $b_1=x_1\partial_2$, and
  $b_2=x_2\partial_1$; $C_e$ acts trivially on $V_0$ and $\rho(b_1)=\epsilon^2b_1$,
  $\rho(b_2)=\epsilon^{-2}b_2$.  If $e$ is odd $\epsilon^2$ is just another primitive $e$th root of
  unity, and
  \begin{displaymath}
    \So(V\otimes_{\Cb} V^*)^{C_e} = (\So(V_0)\otimes_{\Cb} \So(V_1))^{C_e} = \So(V_0)\otimes_{\Cb} \So(V_1)^{C_e}
    =\Cb[\nabla_1, \nabla_2]\otimes_\Cb \Cb[ b_1b_2,b_1^e,b_2^e]. 
  \end{displaymath}
If $e= 2e_1 $ is even, then  $\epsilon^2$ is an $e_1$th primitive root of unity and
$$\So(V\otimes V^* )^{C_e}= \Cb [\nabla_1,\nabla_2]\otimes_{\Cb} \Cb [b_1b_2,b_1^{e_1},b_2^{e_1}].$$

The first assertion in (1) follows since for the map $l$ in (\ref{diff-map}) we have
$l(b_1b_2)=l(\nabla_1\nabla_2+\nabla_1)$, and the second follows since the elements
$(x_1\partial_2)^e=x_1^e\partial_2^e$ and $ (x_2\partial_1)^e=x_2^e\partial_1^e $ belong to
$ \Dc_1\Dc_1^-$.

Now consider (2), so $e= 2 e_1$, and $t_i= l(b_i)$. The elements $t_1^2$ and $t_2^2$ belong to $\Dc_1\Dc_1^-$, while
$t_1t_2=(x_1\partial_2)^{e_1}(x_2\partial_1)^{e_1} = \nabla_1^{e_1}p_{e_1}(\nabla_2)$, and
$ t_2 t_1 = \nabla_2^{e_1}p_{e_1}(\nabla_1)$ (see \Lemma{\ref{lemma:nabla}}) all belong to
$\Ec$. This implies that $\Rc_1$ is generated as a left $\Ec$-module by $1,t_1$ and $t_2$.
\end{proof}
\begin{remark} Since $\nabla_1\nabla_2=x_1x_2\partial_1\partial_2$,
  $p_e(\nabla_1)=x_1^e\partial_1^e$ and $p_e(\nabla_2)$ \Lem{\ref{lemma:nabla}} belong to
  $ \Dc_1^0\cap (\Dc_1 \cdot \Dc_1^-) $, the kernel of $\Pi$ contains
  $ \nabla_1\nabla_2, p_e(\nabla_1)$ and $p_e(\nabla_1)$, implying that $\Rc_1$ is
  artinian. However, we do not know if $\Ker (\Pi)$ is generated by these three
  elements.
\end{remark}

\subsection{The dihedral group}\label{dihedral}

  Now we will consider the dihedral group $G= \Cb_e \rtimes S_2$ and the invariant map $A_2 \to
  B$. In the decomposition of $B$ as $\Dc_2$-module the following  modules will occur:
  \begin{align*}
M_0 & = A_2,\quad  M_e=\Dc_2(x_1^e-x_2^e)=A_2(x_1^e-x_2^e),\\
M_i^1 &=\Dc_2x_1^i=A_2 x_1^i+ A_2x_2^{e-i},\quad  M_i^2=\Dc_2x_2^i=A_2 x_2^i+ A_2x_1^{e-i}, \quad  1 \leq i< e/2,\\
M_{e_1}^I&=\Dc_2(x_1^{e_1}+x_2^{e_1})=A_2(x_1^{e_1}+x_2^{e_1}), \quad M_{e_1}^{II} = \Dc_2(x_1^{e_1}-x_2^{e_1})=A_2(x_1^{e_1}-x_2^{e_1})  ,
  \end{align*}
  where the latter two modules are defined only when $e= 2 e_1$ is an even number, and we notice
  that $M_e$ is generated by the Jacobian $J=x_1^e-x_2^e $ of the invariant map.

\begin{proposition}
\label{prop:dihedral1}Let $A_2=B^G=\Cb [x_1^e+x_2^e,x_1x_2]$.
\begin{enumerate}
\item $B$ has the basis $1,x_1^i,x_2^i, x_1^e-x_2^e,\ \ 1\leq i\leq e-1,$ over $A_2$.
\item
The $\Dc_2$-module $B$ has the following   decomposition into simple components
\begin{displaymath}
  B=A_2\oplus M_e\oplus   (\bigoplus_{1\leq i< e/2} (M_i^1+M_i^2))\oplus M_{e_1}^I \oplus M_{e_1}^{II} 
\end{displaymath}
where the two terms  that  contain $e_1$ only occur when $e= 2e_1$  is an even number. 
\item We have an isomorphism $M_i^1\cong M_i^2$, $1\leq i<e/2$, where the modules have rank 2  over $A_2$,
  and these are the only isomorphisms between the modules in (2). In the case of odd $e$, the only
  $\Dc_2^0$-submodules of $B$ rank $1$ as $A_2$-modules are the non-isomorphic $M_0$ and $M_e$,
  while if $e= 2e_1$ is an even number, then also $M_{e_1}^I$ and $M_{e_1}^{II}$ have rank 1, and
  are non-isomorphic.  
\item In the decompostion of the $\Dc^0_2$-module  $B^{ann}$ the simple components are
  \begin{displaymath}
M^{ann}=\begin{cases}\Cb  x_1^i, & M=M_i^1,\ 0\leq i< e/2,\\
\Cb x_2^{i}, &  M=M_i^2, \ 1\leq  i<e/2,\\
\Cb (x_1^{e}-x_2^{e}), &  M=M_e, \\
\Cb (x_1^{e_1}+x_2^{e_1}), & M=M_{e_1}^I\text{ ($e= 2 e_1$ is  even) },\\
\Cb (x_1^{e_1}-x_2^{e_1}), &M=M_{e_1}^{II}\text{ ($e= 2 e_1$ is  even) }.
\end{cases}
\end{displaymath}
\end{enumerate}
\end{proposition}

\begin{proof}
 \Lemma{\ref{power-inv}} gives the equality $B^G=\Cb[x_1^e+x_2^e,x_1x_2]$, which implies (1).

 We first prove
\begin{displaymath}\tag{*}
  B^{ann} = \Ann_{\Dc_2^-} (B)=  \Cb \cdot 1  +   \sum_{i=1}^{e/2}(\Cb x_1^i+ \Cb x_2^i )  + \Cb( x_1^e-x_2^e).
\end{displaymath}
We have $B^{ann} \subset \Ann_{S(V^*)^G_+}(B)$, where the space of harmonic polynomials (see
\Remark{\ref{harmonics}})
\begin{align*}
  \Ann_{S(V^*)^G_+}(B) &= \{ p \in B \ \vert
  \ \partial_1\partial_2p= (\partial_1^e+ \partial_2^e)p =0 \}\\ &= \Cb \cdot 1 + \sum_{i=1}^{e-1}( \Cb
  x_1^i + \Cb x_2^i) + \Cb( x_1^e-x_2^e).
\end{align*}
To see this, note that the first condition $\partial_1\partial_2p =0$ implies that a harmonic
polynomial $p$ contains no mixed terms, and that $c_1x_1^e + c_2 x_2^e$ is killed by
$ (\partial_1^e+ \partial_2^e)$ only if $c_2= -c_1$.  Let
$p= a_0 + \sum_{i=1}^{e-1} (a_ix_1^i + b_i x_2^i) + c (x_1^e-x_2^e) \in \Ann_{S(V^*)^G_+}(B) $.
Assume that $j < e/2$, so $e-j > j$. Then
$r=x_1^j\partial_2^{e-j}+x_2^j\partial_1^{e-j}\in \Dc_2^-$, so that if $p\in B^{ann}$, then
$r(p)=0$ and therefore
\begin{align*}
  r(p) &= \sum _{i= e-j}^{e-1} a_ir(x_1^i) + b_i r(x_2^i) + cr  (x_1^e-x_2^e)=\sum _{i=
         e-j}^{e-1} a_ir(x_1^i) + b_i r(x_2^i) \\
&= \sum _{i=
         e-j}^{e-1} \frac {i!}{(e-j)!}  (a_ix_2^j x_1^{i-(e-j)}+ b_i x_1^j x_2^{i-(e-j)})=0.
\end{align*}
Since all monomials in the last sum are distinct ($j \neq i-(e-j)$) this implies that $a_i= b_i=0$
when $i > e/2$. Therefore the left side of \thetag{$*$} is a subset of the right side. Conversely,
\Proposition{\ref{prop:cyclic2}} implies that the right side of \thetag{$*$} is contained in
$\Ann_{\Dc_1^-}(B) \subset \Ann_{\Dc_2^-}(B) $, since $\Dc_2^- \subset \Dc_1^-$.

Since  $\Dc_2^0 \subset \Dc_1^0$ and by \Proposition{\ref{prop:cyclic2}} $\Dc_1^0$ preserves the
1-dimensional spaces $\Cb x_1^i$, $\Cb x_2^i$ when $0 \leq i < e/2$, and also the space $\Cb
(x_1^e-x_2^e)$, they define simple $\Dc_2^0$-modules.  When $e= 2 e_1$ is an even number, then
$\Dc_1^0$ and $S_2$ acts on the space  $\Cb x_1^{e_1} + \Cb x_2^{e_1}$, which can be split according
to the $S_2$-action into two   simple $\Dc_2^0$-modules $\Cb (x_1^{e_1}\pm x_2^{e_1}) $. This proves
the decomposition of $B^{ann}$ in (4). 

By \Corollary(\ref{cor:d0eqconcrete}) each of the above simple $\Dc_2^0$-modules generates a simple
$\Dc_2$-submodule of $B$, which implies (2).

It remains to see (3), and for this it suffices, by \Corollary(\ref{cor:d0eqconcrete}), to see that
the assertions hold for the terms in (3).  The element $\sigma \in G$ induces an isomorphism
$\Dc^0_2 x_1^i \cong \Dc_2^0x_2^i$.  Assume first that $e$ is an odd number. Then the eigenspace
decomposition of $B^{ann}$ with respect to the element $\nabla = \nabla_1+ \nabla_2 \in \Dc_2^0 $,
all have multplicity $1$, hence they are non-isomorphic as $\Dc_2^0$-modules as well.  When
$e= 2 e_1$ is an even number, then the $\nabla$-eigenspace
$\Cb (x_1^{e_1}+ x_2^{e_1})\oplus \Cb (x_1^{e_1}- x_2^{e_2}) $ has multiplicity 2, with different
eigenvalue $e_1$ from the other ones.  Now since the element
$ s= (x_1\partial_2)^{e_1} + (x_2\partial_1)^{e_1}$ belongs to $ \Dc_2^0$ and
\begin{displaymath}
  s (x_1^{e_1}+ x_2^{e_1}) = e_1! (x_1^{e_1}+ x_2^{e_1}), \quad   s (x_1^{e_1}- x_2^{e_1}) =   -e_1!
  (x_1^{e_1}- x_2^{e_1}), 
\end{displaymath}
it follows that the two $\Dc_2^0$-modules $\Cb (x_1^{e_1}\pm x_2^{e_1})$ are non-isomorphic.
\end{proof}
\begin{proposition}\label{dihedral-R} If $e$ is an  odd integer,  then the canonical composed homomorphism
$$
\Pi:\Cb[\nabla] \to \Dc_2^0\to  \Rc_2= \frac {\Dc_2^0}{\Dc_2^0 \cap (\Dc_2 \cdot \Dc_2^-)}
$$
is surjective. If $e= 2 e_1$ is an even integer, then $\Rc_2$ is generated as a module over
$\Ec=\Pi(\Cb [\nabla])$ by the elements $1$ and
$(x_1\partial_2)^{e_1}+(x_2\partial_1)^{e_2}\notin \Ec$. In either  case  $\Rc_2$ is commutative.
\end{proposition}
\begin{proof}
  Similarly to the proof of \Proposition{\ref{prop:cyclic3}} we will use the method in
  \Lemma{\ref{lemma:gr}} to compute $\Dc_2^0$. Here the $G$-representation
  $V\otimes_{\Cb} V^*=V_0 \oplus V_1$, where $V_0 = \Cb \nabla_1 + \Cb \nabla_2$ and
  $V_1 = \Cb b_1 + \Cb b_2$, $ b_1=x_1\partial_2$, and $b_2=x_2\partial_1$.  The action of $G$ on
  $V_0$ is through the map $G\to S_2$, and and $\So(V_0)=\So(V_0)^{S_2}\oplus \So(V_0)_\chi$ where
  the character $\chi:G\to\Cb^*$ is given by $\chi(\rho)=1$ and $\chi(\sigma)=-1$.  The action of
  $G=C_e\rtimes S_2$ on $V_1 $ is similar to that on $V$, namely $\rho(b_1)=\epsilon^2$ and
  $\rho(b_2)=\epsilon^{-2}$, $\sigma (b_1)= b_2$, $\sigma (b_2)= b_1$. The invariants are therefore
\begin{align*}\tag{*}
  \So(V\otimes_{\Cb} V^*)^G&= (\So(V_0)\otimes_{\Cb}\So(V_1))^G=  ((\So(V_0)^{G}\otimes_{\Cb}
  S(V_1)^{G})) \oplus (\So(V_0)_\chi\otimes_{\Cb} \So(V_1)_\chi\\
&  = \So(V_0)^G\otimes_\Cb
  \So(V_1)^G  \oplus (\So(V_0)^G\otimes_\Cb
  \So(V_1)^G) \cdot ((\nabla_1-\nabla_2)\otimes (b_1^{e'}-b_2^{e'} )),
\end{align*}
where the last step  follows since the semi-invariants  are given by
$$\So(V_0)_\chi=\So(V_0)^{G}( \nabla_1-\nabla_2), \quad \text{ and }\quad 
\So(V_1)_\chi=\So(V_1)^G (b_1^e-b_2^e).$$
From this we can get  generators of $ \So(V\otimes_{\Cb} V^*)^G$, which  give rise to the following generators of $\Dc_2^0$ (as
described in \Proposition{\ref{prop:cyclic3}}) 
$$1, \nabla , \nabla_1\nabla_2, b_1^{e'}+b_2^{e'}, b_1b_2+b_2b_1 \text{ and } (\nabla_1-\nabla_2) (b_1^{e'}-b_2^{e'} ),$$
where $e'= e_1= e/2$ if $e$ is an even number, and otherwise $e'=e$.  Not all these generators are
required to generate $\Rc_2$.  First we have $\nabla_1 \nabla_2 \in \Dc_2^0 \cap (\Dc_2\Dc_2^-) $,
secondly $b_1b_2+b_2b_1=2\nabla_1\nabla_2+\nabla \equiv \nabla $, where the congruence is modulo
$\Dc_2^0 \cap (\Dc_2\Dc_2^-)$, and third,
\begin{align*}
&  (\nabla_1-\nabla_2) (b_1^{e'}-b_2^{e'} )) =  e'(b_1^{e'}+ b_2^{e'}) + (x_1^{e'+1}\partial_{2}^{e'-1}
                                              + x_2^{e'+1}\partial_1^{e'-1})\partial_1\partial_2 \\ & -
                                              x_1x_2(x_2^{e'-1}\partial_1^{e'+1}
                                              +x_1^{e'-1}\partial_2^{e'+1})
                                             \equiv  e' (b_1^{e'}+ b_2^{e'}).
\end{align*}
Finally, $ b_1^e + b_2^e$ maps to $\Ec$, since
$$
2(b_1^e + b_2^e)=(x_1^e+x_2^e)(\partial_1^e+\partial_2^e)- (x_1^e\partial_1^e+x_2^e\partial_2^e)
\equiv p(\nabla_1, \nabla_2 )
$$
where $p(t_1, t_2)$ is a symmetric polynomial so that
$p(\nabla_1, \nabla_2) \equiv q(\nabla)$ for some polynomial $q(t)$.
It follows that $\Ec = \Rc_2$ when $e'=e$ is odd. If $e'= e_1= e/2 $ and $e$ is even, the residue
class $\bar s$ of $s = b_1^{e_1} + b_2^{e_1} \in \Dc^0_2$ does not belong to $\Ec$.  This follows
since $\nabla$ acts as a constant on the 2-dimensional space
$\Cb (x_1^{e_1}+ x_2^{e_1}) + \Cb (x_1^{e_1} - x_2^{e_1})\subset B^{ann}$, while $s$ splits this
space into two eigenspaces with distinct eigenvalues; see the end of the proof of
\Proposition{\ref{prop:dihedral1}}. Since $[\nabla, \Dc_2^0]=0 $ it follows that
$\bar s \Ec \subset \Ec \bar s$, so it remains to prove that $\bar s^2\in \Ec$. This follows since
\begin{displaymath}
  s^2 =  b_1^e + b_2^e
  +x_1^{e'}\partial_2^{e'}x_2^{e'}\partial_1^{e'}+x_2^{e'}\partial_1^{e'}x_1^{e'}\partial_2^{e'}
\end{displaymath}
where we already have seen that the projection of $ b_1^e + b_2^e$ in $\Rc_2$ belongs to $\Ec$ and
the projection of the remaining part of the right side also belongs to $\Ec $ since
$x_1^{e'}\partial_2^{e'}x_2^{e'}\partial_1^{e'}+x_2^{e'}\partial_1^{e'}x_1^{e'}\partial_2^{e'} \in
\Cb[\nabla_1, \nabla_2]^{S_2}$,
and also recalling that $\nabla_1\nabla_2 \in \Dc_2^0\cap (\Dc_2\cap \Dc_2^-) $.
\end{proof}

\begin{remark} For $e=1$ the dihedral group is just $S_2$, so in that
  case $\Rc$ is actually generated by just $\nabla$ (compare
  \Proposition{\ref{gen-sym-R}}). For $e=3$, the dihedral group is again
  a symmetric group $S_3$, but now acting on $\Cb^2$, as compared to
  the $3$-dimensional permutation representation studied in
  \Proposition{\ref{gen-sym-R}}. 
\end{remark}

\subsection{MLS-restriction}
We can now see precisely how the MLS-restriction behaves between the cyclic group $C_e$ and the
dihedral group $D_{2e}= A(e,e,2)\rtimes S_2$, by considering how the simple modules over $\Dc_1^0$ in
\Proposition{\ref{prop:cyclic2}}  behave on restriction to the subring
$\Dc_2^0$:

\begin{corollary} Put  $J=\Jo_{C_e}^{D_{2e}}$. 
\label{lemma:mlsgeneralized}
\begin{enumerate}
\item   $\Jo(N_i)=M_i^1$, for $0\leq i<e/2$ and $\Jo(N_i)=M_{e-i}^2$, for $e/2<i<e$,
\item $\Jo (N_{e_1})=M_{e_1}^I\oplus M_{e_1}^{II}$ ($e= 2 e_1$ is even).
\item The simple $\Dc_2$-submodules of $B$ are either isomorphic to the image under MLS-restriction
  of a simple   $\Dc_1$-module, or are modules of semi-invariants belonging to a linear
  character of $D_{2e}$. The latter case occurs for $M_{e_1}^I$, $M_{e_1}^{II}$ ($e= 2 e_1$ is  even) and
  $M_e$ (all $e)$.
\end{enumerate}
\end{corollary}

\begin{proof} This follows from \Propositions{\ref{prop:cyclic2} }{\ref{prop:dihedral1}}.

\end{proof}
                
\end{document}